\documentclass[twoside, letterpaper, 12pt]{article}

\DeclareMathSizes{10}{10}{10}{10}
\usepackage{anysize}
\usepackage{fancyhdr}
\marginsize{2cm}{2cm}{0.5cm}{0.75cm}
\usepackage{setspace}
\usepackage[utf8]{inputenc} % allow utf-8 input
\usepackage[T1]{fontenc}    % use 8-bit T1 fonts
\usepackage{url}            % simple URL typesetting
\usepackage{booktabs}       % professional-quality tables
\usepackage{amsfonts}       % blackboard math symbols
\usepackage{nicefrac}       % compact symbols for 1/2, etc.
\usepackage{microtype}      % microtypography
\usepackage{amsthm}
\usepackage{caption}
\usepackage{subcaption}
\usepackage{graphics} % for pdf, bitmapped graphics files
\usepackage{epsfig} % for postscript graphics files
\usepackage{times} % assumes new font selection scheme installed
\usepackage{amsmath} % assumes amsmath package installed
\usepackage{amssymb}  % assumes amsmath package installed
\usepackage{makeidx}
\usepackage{float}
\usepackage{balance}
\usepackage{booktabs}
\usepackage{bbm}
\usepackage{mathrsfs}
\usepackage{enumitem}
\usepackage{multicol}
\usepackage{algorithm}
\usepackage{algpseudocode}

\usepackage{mathrsfs}  
\usepackage[normalem]{ulem}
\usepackage{xr} % enable referencing labels in other documents
\usepackage{multicol}
\usepackage{float}
\usepackage[nolist]{acronym}
\usepackage{xcolor}
\usepackage[colorlinks]{hyperref}       % hyperlinks

\hypersetup{
  colorlinks = true,
  allcolors = siaminlinkcolor,
  urlcolor = siamexlinkcolor,
}
\colorlet{siaminlinkcolor}{green!50!black}
\colorlet{siamexlinkcolor}{red!50!black}

\usepackage{lipsum} %this package can be removed

\usepackage{setspace}
\usepackage{cleveref} %Cleveref must be defined after hyperref

\newcommand{\rd}{{\mathrm d}}

\newcommand{\rP}{{\mathrm{P}}}

\newcommand{\vx}{{\bf x}}
\newcommand{\vy}{{\bf y}}

\newcommand{\vM}{{\bf M}}

\newcommand{\vPhi}{{\mbox{\boldmath$\Phi$}}}

\newcommand{\calB}{{\cal B}}

\newcommand{\calE}{{\cal E}}
\newcommand{\calF}{{\cal F}}

\newcommand{\calH}{{\cal H}}

\newcommand{\calL}{{\cal L}}

\newcommand{\calN}{{\cal N}}

\newcommand{\calP}{{\cal P}}

\newcommand{\calU}{{\cal U}}

\newcommand{\argmin}{\operatornamewithlimits{argmin}}

\newcommand{\Eb}{\mathbb{E}}

\newcommand{\Pb}{\mathbb{P}}
\newcommand{\Qb}{\mathbb{Q}}
\newcommand{\Rb}{\mathbb{R}}

\newcommand{\Zb}{\mathbb{Z}}

\newcommand{\F}{\mathscr{F}}

\newcommand{\A}{\mathscr{A}}

\newcommand{\vm}{{\bf m}}

\newcommand{\vTheta}{{\mbox{\boldmath$\Theta$}}}

\makeatletter
\renewcommand{\maketag@@@}[1]{\hbox{\m@th\normalsize\normalfont#1}}%
\makeatother

\newtheorem{theorem}{Theorem}[section]

\newtheorem{prop}[theorem]{Proposition}

\newtheorem{definition}{Definition}[section]

\catcode`\_=11\relax
    \newcommand\email[1]{\_email #1\q_nil}
    \def\_email#1@#2\q_nil{%
      \href{mailto:#1@#2}{{\emailfont #1\emailampersat #2}}
    }
    \newcommand\emailfont{}
    \newcommand\emailampersat{\small@}
    \catcode`\_=8\relax 
    
\makeatletter
\newcommand{\vast}{\bBigg@{3.5}}
\newcommand{\Vast}{\bBigg@{4}}
\newcommand{\vastt}{\bBigg@{4.5}}
\newcommand{\Vastt}{\bBigg@{5}}
\makeatother

\newcommand{\beginsupplement}{
        \setcounter{table}{0}
        \renewcommand{\thetable}{S\arabic{table}}
        \setcounter{figure}{0}
        \renewcommand{\thefigure}{S\arabic{figure}}
        \setcounter{section}{0}
        \renewcommand{\thesection}{S\arabic{section}}
        \setcounter{equation}{0}
        \renewcommand{\theequation}{S\arabic{equation}}
        \setcounter{algorithm}{0}
        \renewcommand{\thealgorithm}{S\arabic{algorithm}}
     }

\setlist[enumerate]{
  label={\upshape(\roman*)},
  leftmargin=30pt,
  labelwidth=0pt,
  align=right,
  nosep
}

\begin{acronym}
\acro{DNN}{Deep Neural Network}
\acro{ODE}{Ordinary Differential Equation}
\acro{SPDE}{Stochastic Partial Differential Equation}
\acro{FNN}{Feed-forward Neural Network}
\acro{CNN}{Convolutional Neural Network}
\acro{DP}{Dynamic Programming}
\acro{LSTM}{Long-Short Term Memory}
\acro{FC}{Fully Connected}
\acro{DDP}{Differential Dynamic Programming}
\acro{HJB}{Hamilton-Jacobi-Bellman}
\acro{PDE}{Partial Differential Equation}
\acro{PI}{Path Integral}
\acro{NN}{Neural Network}
\acro{SOC}{Stochastic Optimal Control}
\acro{RL}{Reinforcement Learning}
\acro{MPC}{Model Predictive Control}
\acro{IL}{Imitation Learning}
\acro{RNN}{Recurrent Neural Network}
\acro{DL}{Deep Learning}
\acro{SGD}{Stochastic Gradient Descent}
\acro{SDE}{Stochastic Differential Equation}
\acro{VRL}{Variational Reinforcement Learning}
\acro{IDVRL}{Infinite Dimensional Variational Reinforcement Learning}
\acro{1D}{1-dimensional}
\acro{2D}{2-dimensional}
\acro{ROM}{Reduced Order Model}
\acro{SDE}{Stochastic \ac{ODE}}
\acro{ANN}{Artificial Neural Network}
\acro{ADPL}{Actuator Design and Policy Learning}
\end{acronym}

\begin{document}
	
	\title{
	\rule{\linewidth}{4pt}\vspace{0.3cm} \Large \textbf{
	  Spatio-Temporal Stochastic Optimization: \\ Theory and Applications to Optimal  Control and Co-Design 
	}\\ \rule{\linewidth}{1.5pt}}
	\author{Ethan N. Evans$^{a,}\thanks{Corresponding Author. Email: \email{eevans41@gatech.edu}}~$, Andrew P. Kendall$^{a}$, George I. Boutselis$^{a}$, \\ and Evangelos A. Theodorou$^{a,b}$\\ \vspace{-0.1cm}
	\small{$^a$Georgia Institute of Technology, Department of Aerospace Engineering} \\ \vspace{-0.2cm}
	\small{$^b$Georgia Institute of Technology, Institute of Robotics and Intelligent Machines} }
	
	\date{\small{This manuscript was compiled on \today}}
	
	\maketitle

\begin{abstract}
There is a rising interest in Spatio-temporal systems described by Partial Differential Equations (PDEs) among the control community. Not only are these systems challenging to control, but the sizing and placement of their actuation is an NP-hard problem on its own. Recent methods either discretize the space before optimziation, or apply tools from linear systems theory under restrictive linearity assumptions. In this work we consider control and actuator placement as a coupled optimization problem, and derive an optimization algorithm on Hilbert spaces for nonlinear PDEs with an additive spatio-temporal description of white noise. We study first and second order systems and in doing so, extend several results to the case of second order PDEs. The described approach is based on variational optimization, and performs joint RL-type optimization of the feedback control law and the actuator design over episodes. We demonstrate the efficacy of the proposed approach with several simulated experiments on a variety of SPDEs.
\end{abstract}

%===============================================================================

\section{Introduction}\label{sec:intro}
In many complex natural processes, a variable such as temperature or displacement has values that are time varying on a spatial continuum over which the system is defined. These spatio-temporal processes are typically described by \acp{PDE} and are increasingly prevalent throughout the robotics community. Swarm robotics can be described by Reaction-advection-diffusion \acp{PDE}~\cite{elamvazhuthi2018pde}. Robot navigation in crowded environments can be described by Nagumo-like \acp{PDE}~\cite{aidman2008coupled}. Soft robotic limbs can be modelled as damped Euler-Bernoulli systems~\cite{shapiro2015modeling}.

Some of the major control-related challenges of these systems include dramatic under-actuation, high system dimensionality, and the design and/or placement of distributed actuators over a continuum of potential locations. These systems often have significant time delay from a control signal, and can have several bifurcations and multi-modal instabilities. In addition, realistic representations of these systems are stochastic. Put together, control of spatio-temporal systems represents many of the largest current-day challenges facing the robotics and automatic control community.

This paper addresses stochastic optimal control and co-design of \acp{SPDE} through the lens of stochastic optimization. We propose a joint actuator placement and policy network optimization strategy via episodic reinforcement that leverages inherent spatio-temporal stochasticity in the dynamics for optimization. The resulting stochastic gradient descent approach bootstraps off the widespread success of SGD methods such as ADAM for training \acp{ANN}. 

Among the main goals of this line of research is to establish capabilities for the eventual design and manufacture of soft-body robots. The behavior of such systems follows second order \acp{SPDE}. As such, while the proposed method is general to first and second order systems, we focus our mathematical formulation on second order \ac{SPDE}.
\newline

\noindent\textbf{Our Contributions.} 
In this work we tackle the coupled challenge of policy optimization and actuator co-design for \acp{SPDE}. Our approach is founded on a general principle coming from thermodynamics that also has had success in stochastic optimal control literature~\cite{TheodorouCDC2012}
\begin{equation}\label{eq:Free_Energy_Relative_Entropy}
\text{Free Energy} \leq \text{Work} - \text{Temperature} \times \text{Entropy}.
\end{equation}
We leverage this principle in order to derive a measure-theoretic loss function that utilizes exponential averaging over importance sampled system trajectories in order to choose network and actuator design parameters that simultaneously minimize state cost and control effort. This work builds off related work \cite{Evans2019IDVRL} by including actuator co-design and establishing needed theory for applicability to second order \acp{SPDE}. Specifically, we contribute the following:
\newline
\begin{enumerate}
    \item A derivation of a joint policy optimization and actuator co-design architecture in Hilbert spaces for \acp{SPDE}
    \item A version of the Girsanov Theorem suitable for second-order \acp{SPDE} in a natural form
    \item A practical set of tools to extend related work to actuator co-design and to handle second-order SPDEs
\end{enumerate}

%===============================================================================

\section{Related Work}
\noindent\textbf{Control of \acp{SPDE}.} There is a growing body of work that seek control of \acp{PDE} by immediately reducing them to a set of ODEs~\cite{bieker2019deep,nair2019cluster,mohan2018deep,morton2018deep,rabault2019artificial}. They do not consider stochasticity and typically use standard tools from finite-dimensional machine learning and control theory. Other approaches apply \ac{SPDE} control theory, yet either treat linear systems or lack numerical results~\cite{StochasticBurgers_1999, moura2013optimal}. We build the proposed framework on \cite{Evans2019IDVRL}, wherein the authors create a semi-model-free reinforcement learning framework for policy-based control of \acp{SPDE}. This related work does not consider actuator co-design and does not give a mathematical treatment suitable to second-order \acp{SPDE}. 

% Numerical results and algorithms for distributed control of \acp{SPDE}  are limited and typically require some model reduction approach~\cite{lou2009model,gomes2017controlling}.  In~\cite{StochasticBurgers_1999}, the authors approach the control of the stochastic Burgers equation through the Hamilton-Jacobi-Bellman theory by applying the linear Feynman-Kac lemma; nevertheless, it lacks numerical results. In~\cite{moura2013optimal}, the authors treat optimal control of linear deterministic \acp{PDE} by applying linear control theory, however this work is limited to linear \acp{PDE}.

\vspace{1em}
\noindent\textbf{Actuator Co-Design for \acp{PDE}.}  In \cite{yang2017optimal}, theoretical results for the optimal actuator location for minimum norm control of the stochastic heat equation are obtained, however, without a numerical demonstration. Optimal co-design techniques for actuator and sensor locations with $H_\infty$ and $H_2$ design objectives were demonstrated on flexible structures \cite{lim1992method,nestorovic2013optimal,kasinathan2013h}, and the linearized Ginzburg-Landau equation \cite{chen2011h,manohar2018optimal}, but are only applicable to linear systems. Similarly, symmetry properties and linear control theory are applied to chaotic systems in \cite{grigoriev1997pinning}. Optimal actuator and sensor placement for scalar control of the linear advection equation are demonstrated in \cite{sinha2013optimal,vaidya2012actuator} using Gramian based approaches which also requires system linearity. Similarly, \cite{amstutz2006new} provides a level set method that promises scalability, yet relies on linear system Grammians. Conditions for the existence of optimal actuator placement for semilinear PDEs are obtained in \cite{edalatzadeh2019optimal}.  
Optimal actuator and sensor placement for stabilization of the nonlinear Kuramoto-Sivashinsky equation is demonstrated in \cite{lou2003optimal}, however they impose strong simplifying assumptions which limit their dimensionality. 

%===============================================================================

\section{Problem Formulation}
\label{sec:formulation}

%===============================================================================
% Zakai: $\rd u = (\frac{1}{2}\partial_{xx} (a u) - \partial_x(bu))\rd t -\partial_x (\gamma u) \rd W$
This work proposes actuator co-design optimization for a large class of stochastic spatio-temporal systems represented as \acp{SPDE}. We describe these systems as evolving on time separable Hilbert spaces, where they are represented by infinite dimensional vectors and acted on by operators. We address a class of \acp{SPDE} that are of \textit{semi-linear} form. Let $\calH$ denote a separable Hilbert space with $\sigma$-field $\calB(\calH)$ and probability space $(\Omega, \calF, \rP)$ with filtration $\calF_t$, $t\in [0,T]$. Consider the general semi-linear form of a controlled \ac{SPDE} on $\calH$ given by 
\begin{align}\label{eq:SPDEs_Control}
\rd X = \big( \A X  + F(t, X)  \big) \rd t + G(t, X)\big(\vPhi(t,X,\vx; \vTheta^{(k)})\rd t+ \frac{1}{\sqrt{\rho}} \rd  W(t)\big), 
\end{align}
where $X(t) \in \calH$ is the state of the system which evolves on the Hilbert space $\calH$, the linear and nonlinear measurable operators $\A: \calH \rightarrow \calH$ and $F(t,X): \Rb \times \calH \rightarrow \calH$  (resp.) are uncontrolled drift terms, $\vPhi(t,X,\vx; \vTheta^{(k)}):\Rb \times \calH \times D \rightarrow \calH$ is the nonlinear control policy parameterized by $\vTheta^{(k)}$ at the $k^{th}$ iteration, where $D \subset \Rb^3$ is the domain of the finite spatial region, $\rd W(t):\Rb \rightarrow \calH$  is a spatio-temporal noise process, and $G(t,X)$ is a nonlinearity that affects both the noise and the control. It is used to incorporate the effects of actuation on the field. 

The \textit{Hilbert spaces} formulation given in \cref{eq:SPDEs_Control} is general in that any semi-linear \ac{SPDE} can be described in this form by appropriately choosing the $\A$ and $F$ operators. In this form, the spatio-temoral noise process $\rd W(t)$ is a Hilbert space-valued Wiener process, which is a generalization of the Wiener process in finite dimensions. We include a formal definition of a Wiener process in Hilbert spaces for clarity \cite[Section 4.1.1]{da2008stochastic}
\begin{definition} A $\calH$-valued stochastic process $W(t)$ is called a Wiener process if
\begin{enumerate}
    \item $W(0) = 0$
    \item $W$ has continuous trajectories
    \item $W$ has independent increments
    \item $\calL\big(W(t) - W(s)\big) = \calN (0, t-s)Q), \quad t \geq s \geq 0$
    \item $\calL\big(W(t)\big) = \calL\big(-W(t)\big), \quad t \geq 0$
\end{enumerate}
\end{definition}

\begin{prop}
Let $ \{e_{i}\}_{i=1}^{\infty} $  be a complete orthonormal system for the Hilbert Space $\calH$. Let $Q$ denote the covariance operator of the Wiener process $W(t)$. Note that $Q$ satisfies $Q e_{i}= \lambda_{i} e_{i}$, where $\lambda_{i}$ is the eigenvalue of $Q$ that corresponds to eigenvector $e_{i}$. Then, $W(t) \in \calH$ has the following expansion:
 \begin{equation}\label{eq:Wiener_expansion}
    W(t) = \sum_{j=1}^{\infty} \sqrt{\lambda_{j}} \beta_{j}(t) e_{j},
\end{equation}
\noindent where  $ \beta_{j}(t)  $  are real valued Brownian motions that are mutually independent on $ (\Omega, \F, \mathbb{P})$.
\end{prop}
This expansion in \cref{eq:Wiener_expansion} reveals how the Wiener process acts spatially. There are various forms of the Wiener process with different properties. We refer the interested reader to \cite{da2008stochastic} for a more complete introduction. The proposed approach is derived for a special case of Wiener process called the Cylindrical Wiener process, defined as follows.
\begin{definition}
A Wiener process $W(t)$ on $\calH$ is called a Cylindrical Wiener process if the covariance operator $Q$ is the identity operator $I$.
\end{definition}

Note that for the Cylindrical Wiener process, the sum in \cref{eq:Wiener_expansion} is unbounded in $\calH$ since $\lambda_j = 1$, $\forall j=1,2,\dots$. %However, in this case the series converges in another Hilbert space $\calH_{1}\supset \calH$ when the inclusion $\iota:U\rightarrow U_{1}$ is Hilbert-Schmidt. For more details see \cite{da1992stochastic}. 
This makes the Cylindrical Wiener process a challenging Wiener process to handle since it acts spatially  \textit{everywhere} with \textit{equal} magnitude, in contrast to Wiener processes with covariance operators that are of trace class (i.e. wherein the expansion \cref{eq:Wiener_expansion} is finite). This type of spatio-temporal noise requires the assumption that the operators $\A$, $F(t, X)$, and $G(t,X)$ satisfy properly formulated conditions given in \cite[Hypothesis 7.2]{da2008stochastic} to guarantee the existence and uniqueness of the $\calF_t$-adapted weak solution $X(t)$, $t \geq 0$.
% \begin{assumption}
% Consider the system in \cref{eq:SPDEs_Control}, with Cylindrical Wiener Process $W(t)$ on $\calH$. Assume the operators $\A$, $F(t, X)$, and $G(t,X)$ satisfy properly formulated conditions given in \cite[Hypothesis 7.2]{da2008stochastic} to guarantee the existence and uniqueness of the $\calF_t$-adapted weak solution $X(t)$, $t \geq 0$.
% % Assume that the operator $\A$ is properly contractive such that
% % \begin{equation}
% %     \int_{0}^{T} \vert\vert e^{(T-s)\A} G\big(s, X(s)\big)\vert\vert_{L_2^0}^2 \rd s < \infty
% % \end{equation}
% \end{assumption}
% This assumption is required to properly define the stochastic integral given by \cite[Section 4.2]{da1994stochastic}
% \begin{equation} \label{eq:stoch_integral}
% \int_{0}^{t}e^{(t-s)\A}G\big(s, X(s)\big)\rd W(s)
% \end{equation}

The nonlinear policy $\vPhi(t,X,\vx; \vTheta^{(k)})$ is a potentially time-varying policy that has explicit state dependence. Nonlinear, explicit state dependence allows for a feedback policy that can extract pertinent information from the state for control. In this work the nonlinear policy utilizes an \ac{ANN}. Embedded in this function is also a spatial dependence. The dependence on $\vx$ describes how the actuator is placed in the spatial domain. This approach encompasses cases where terms that parametrize how the actuators are shaped or sized are included in the nonlinear policy.

Many complex spatio-temporal systems are given by partial differential equations of second order in time. One such system is the simply supported stochastic Euler-Bernoulli equation with Kelvin-Voigt and viscous damping, which can describe the motion of a soft robotic limb. Formally, this is given by
\begin{subequations} \label{eq:EB_SPDE}
\begin{align} 
&\partial_{tt} y + \partial_{xx} \big( \partial_{xx}y + C_d \partial_{xxt} y \big) + \mu \partial_t y = \vPhi + \frac{1}{\sqrt{\rho}} \partial_t W(t,x) \\
&y(t,0) = y(t,a) = 0, \\
&y(0,x) = y_0, \\
&\partial_t y(0,x) = v_0 \\
&\partial_{xx}(t,0) + C_d \partial_{xxt}y(t,0) = 0, \\
&\partial_{xx}(t,a) + C_d \partial_{xxt} y(t,a) = 0,
\end{align}
\end{subequations}
where the spatial region is one-dimensional, $y(x,t)=y: \Rb^n \times \Rb \rightarrow \Rb$ represents the vertical displacement of the beam, and all functional dependencies of the nonlinear policy $\vPhi$ have been dropped since it has a different form in the PDE perspective. With the change of variables $v := \partial_t y $, this system has the typical second order matrix form
\begin{align} \label{eq:EB_matrix_form}
     \partial_t \left[\begin{array}{c} y \\ v \end{array} \right] = & \left[\begin{array}{cc} 0 & 1 \\ -A_0 & -C_d A_0 - \mu \end{array} \right] \left[ \begin{array}{c} y \\ v \end{array} \right ] + \left[ \begin{array}{c} 0 \\ 1 \end{array} \right] \Phi  + \left[ \begin{array}{c} 0 \\ \frac{1}{\sqrt{\rho}} \end{array} \right] \partial_t W(t,x) 
\end{align}
where $A_0 = \partial_{xxxx}$ without boundary conditions. Now, we lift this PDE into infinite dimensional Hilbert spaces. Define $Y \in \calH$ as the Hilbert space analog of $y(x,t)$, $V \in \calH$ as the Hilbert space analog of $v(x,t)$, and a variable $Z$ on the direct product Hilbert space $\calH^2 := \calH \times \calH$. Note that $Z$ is a Hilbert space analog of a variable $z(x,t) = [y(x,t) \;\; v(x,t)]^\top \in \Rb^2$. In Hilbert spaces, $A_0$ becomes an operator acting on $\calH$ and $1$ gets replaced by the identity operator $I$ acting on $\calH$. Rewriting \cref{eq:EB_matrix_form} in Hilbert space semi-linear form yields
\begin{equation} \label{eq:EB_Hilbert_form}
    \rd Z = A Z \rd t + G\Big( \vPhi(t,Z,\vx; \vTheta^{(k)})\rd t + \frac{1}{\sqrt{\rho}}\rd W(t) \Big)
\end{equation}
where $A: \calH^2 \rightarrow \calH^2$ is the linear operator $A = [0 \;\; I; -A_0 \;\; -C_d A_0 - \mu I]$, $G: \calH \rightarrow \calH^2$ is an operator representing how control and spatio-temporal noise enter the system $G = [0; I]$, and $\rd W(t)$ is a Cylindrical Wiener process on $\calH$. Note that the Hilbert space variables $Y$, $V$, and $Z$ no longer have spatial dependence as the Hilbert space vectors capture the spatial continuum over which the problem is defined.

\section{A Girsanov Theorem for Second Order SPDEs}
In order to derive a measure theoretic view of variational optimization, we require a change of measures coming from an appropriate Girsanov theorem for second order Hilbert space valued systems. We present the second order version of the Girsanov theorem here as an extension of existing formulations, where all norms and inner products are taken with respect to the Hilbert space $\calH$.

\begin{theorem}[Girsanov] \label{girs} Let $\Omega$ be a sample space with a $\sigma$-algebra $\mathcal{F}$. Consider the following $\calH^2$-valued nonlinear stochastic processes
\begin{align}
\rd Z&=\big(\A Z+F(t, Z)\big) \rd t +  \frac{1}{\sqrt{\rho}}G(t, Z)\rd W(t), \label{eq:Z}\\
\rd\tilde{Z}&=\big(\A \tilde{Z}+F(t, \tilde{Z})\big)\rd t+G(t, \tilde{Z})\bigg(\tilde{B}(t, \tilde{Z})\rd t + \frac{1}{\sqrt{\rho}} \rd W(t)\bigg),\label{eq:Z_tilde}
\end{align}
where $Z(0)=\tilde{Z}(0)=z_0$ and $W \in \calH$ is a Cylindrical Wiener process with respect to measure $\mathbb{P}$. Moreover, let $\Gamma$ be a set of continuous-time, infinite-dimensional trajectories in the time interval $[0,T]$. Now the {\it probability law} of $Z$ will be defined as $\mathcal{L}(\Gamma):=\mathbb{P}(\omega\in\Omega|Z(\cdot,\omega)\in\Gamma)$ . Similarly, the law of $\tilde{Z}$ is defined as $\tilde{\calL}(\Gamma):=\mathbb{P}(\omega\in\Omega|\tilde{Z}(\cdot,\omega)\in\Gamma)$. Then
\begin{equation} \label{eq:Lg}
\begin{split}
\tilde{\calL}(\Gamma) =  \mathbb{E}_{\mathbb{P}}\Bigg[ \exp\bigg(&\int_{0}^{T} \big\langle\psi(s),\rd W(s)\big\rangle -\frac{1}{2}\int_{0}^{T}\big|\big|\psi(s)\big|\big|^{2}\rd s\bigg) \Bigg| Z(\cdot)\in\Gamma\Bigg],
\end{split}
\end{equation}
where we have defined
\begin{equation}
    \psi(t):=\sqrt{\rho}\tilde{B}\big(t, Z(t)\big)\in \calH,
\end{equation}
and assumed
%\begin{equation}
 %   \label{psi_assum1}
  $  \mathbb{E}_{\mathbb{P}}\big[e^{\frac{1}{2}\int_{0}^{T}||\psi(t)||^2\mathrm dt}\big]<+\infty$.
%\end{equation}% Here, we write for brevity $\tilde{\calL}(\omega)\equiv\tilde{\calL}(\tilde{X}(\cdot,\omega)\in\Gamma)$.
\end{theorem}
\begin{proof}
Define the process
\begin{equation}
\label{eq:w_hat}
\hat{W}(t):= W(t)-\int_{0}^{t}\psi(s)\rd s.
\end{equation}
Under the above assumption, $\hat{W}$ is a Cylindrical Wiener process with respect to a measure $\Qb$ defined by
\begin{equation}
\label{eq:girsanov_measure}
\begin{split}
\rd \Qb (\omega)&=\exp\bigg(\int_{0}^{T}\big\langle\psi(s),\rd W(s)\big\rangle-\frac{1}{2}\int_{0}^{T}\big|\big|\psi(s)\big|\big|^{2}\rd s\bigg)\rd\mathbb{P} \\ &=\exp\bigg(\int_{0}^{T}\big\langle\psi(s),\rd \hat{W}(s)\big\rangle+\frac{1}{2}\int_{0}^{T}\big|\big|\psi(s)\big|\big|^{2}\rd s\big)\rd\mathbb{P}.
\end{split}
\end{equation}
The proof for this intermediate result can be found in \cite[Theorem 10.14]{da1992stochastic}. Now, using \cref{eq:w_hat}, \cref{eq:Z} is rewritten as
\begin{align}
\rd Z &= \big(\A  X+F(t, Z)\big)\rd t+\frac{1}{\sqrt{\rho}}G(t, Z)\rd W(t) \label{eq:Z_new0}  \\
     &=\big(\A  Z+F(t, Z)\big)\rd t+G(t, Z)\bigg(B(t,Z)\rd t+\frac{1}{\sqrt{\rho}}\rd\hat{W}(t) \bigg) \label{eq:Z_new1}
\end{align}
Notice that  the  SPDE in \cref{eq:Z_new1} has the same form as \cref{eq:Z_tilde}. Therefore, under the introduced measure $\Qb$ and noise profile $\hat{W}$, $Z(\cdot, \omega)$ becomes equivalent to $\tilde{Z}(\cdot, \omega)$. Conversely, under measure $\mathbb{P}$, \cref{eq:Z_new0} (or \cref{eq:Z_new1}) behaves as the original system in \cref{eq:Z}. In other words, \cref{eq:Z} and \cref{eq:Z_new1} describe the same system on $(\Omega, \mathcal{F}, \mathbb{P})$. From the uniqueness of solutions and the aforementioned reasoning, one has
\[\Pb\big(\{\tilde{Z}\in\Gamma\}\big) = \Qb\big(\{Z\in\Gamma\}\big).\]
The result follows from \cref{eq:girsanov_measure}.
\end{proof}

It follows that the {\it Randon-Nikodym} (RN) derivative between measures $\cal L(\cdot)$ and $\tilde{\calL}(\cdot)$ of the different dynamical systems defined in \cref{eq:Z,eq:Z_tilde} is given by
\begin{equation} \label{eq:RN}
\begin{split}
\frac{\rd \calL}{\rd \tilde{\calL}} = \exp\bigg(-\int_{0}^{T}\big\langle\psi(s), \rd W(s)\big\rangle-\frac{1}{2}\int_{0}^{T}||\psi(s)||^{2}\rd s\bigg).
\end{split}
\end{equation}
Applying this to both the general semi-linear system \cref{eq:SPDEs_Control} and the second order semi-linear form of the Euler-Bernoulli system \cref{eq:EB_Hilbert_form} yields $\psi:=\sqrt{\rho}\vPhi(t,Z, \vx; \vTheta^{(k)})$ and a Radon-Nikodym derivative as
\begin{equation} \label{eq:RN_semi_linear}
\begin{split}
\frac{\rd \calL}{\rd\tilde{\calL}} = \exp\bigg(-\sqrt{\rho}\int_{0}^{T}\big\langle\vPhi(t,Z,\vx; \vTheta^{(k)}), \rd W(s)\big\rangle -\frac{\rho}{2}\int_{0}^{T}||\vPhi(t,Z,\vx; \vTheta^{(k)})||^{2}\rd s\bigg).
\end{split}
\end{equation}
Throughout the rest of the manuscript, we refer to the terms in the Radon-Nikodym derivative many times. For convenience, they will be assigned functions $\calN(\vTheta, \vx) := \int_{0}^{T}\big\langle\vPhi(t,Z,\vx; \vTheta^{(k)}), \rd W(s)\big\rangle$, and $\calP(\vTheta, \vx) := \int_{0}^{T}||\vPhi(t,Z,\vx; \vTheta^{(k)})||^{2}\rd s$.

%===============================================================================

\section{Spatio-Temporal Stochastic Optimization}

The proposed measure theoretic framework is based on \cref{eq:Free_Energy_Relative_Entropy} in the following form~\cite{theodorou2015nonlinear,theodorou2018linearly}
\begin{align} \label{eq:Legendre}
  & - \frac{1}{\rho}   \log \Eb_{\calL} \bigg[ \exp( -\rho {J} )  \bigg]  = \min_{\calU(\cdot,\cdot)} \bigg[    \Eb_{\tilde{\calL}}\left({J} \right)  + \frac{1}{\rho} D_{KL} ( \tilde{\calL}\; \big|\big|  \calL )  \bigg],
\end{align}
where $J=J(X)$ is an arbitrary state cost function. Relating \cref{eq:Legendre} to \cref{eq:Free_Energy_Relative_Entropy}, the metaphorical work and entropy describe a metaphorical energy landscape for which there is a minimizing measure. Sampling from this measure would simultaneously minimize state cost and the $KL$-divergence term, which is interpreted as control effort. The measure that optimizes \cref{eq:Legendre} is the so-called Gibbs measure
\begin{equation}\label{eq:Gibbs}
\rd \calL^{*} = \frac{\exp( - \rho J) \rd \calL}{\Eb_\calL \big[\exp( - \rho J) \big] }.
\end{equation}

It is not known how to sample directly from the Gibbs measure in \cref{eq:Gibbs}. Instead, variational optimization methods seek to iteratively minimize the controlled distribution's distance to the Gibbs measure \cite{theodorou2018linearly, boutselis2019variational, Williams2016AggressiveDW}. Define the control policy and actuator co-design problem as
\begin{subequations}\label{eq:theta_and_x}
\begin{align}
        \vTheta^{*} &=  \argmin_{\vTheta}  D_{KL}(\calL^{*}|| \tilde{\calL}) \\
        \vx^{*} &= \argmin_{\vx} D_{KL}(\calL^{*}|| \tilde{\calL})
\end{align}
\end{subequations}

Throughout experiments, the authors found that a joint optimization problem dramatically outperforms the split problem in \cref{eq:theta_and_x}. To make this clear, define a new variable $\hat{\vTheta} := [\vTheta;\; \vx]^\top$, and with it the new joint variational optimization as
\begin{equation}
    \hat{\vTheta}^{*} = \argmin_{\hat{\vTheta}} D_{KL}(\calL^{*}|| \tilde{\calL}).
\end{equation}
Expanding the KL divergence and performing importance sampling yields
\begin{align}
    \hat{\vTheta}^{*} &= \argmin_{\hat{\vTheta}} \bigg[  \int \log \Big( \frac{ \rd \calL^*}{\rd \calL} \frac{\rd \calL}{\rd \tilde{\calL}} \Big)  \rd  \calL^* \bigg],
\end{align}
which is equivalent to minimizing
\begin{align}
    \hat{\vTheta}^{*} &= \argmin_{\hat{\vTheta}} \bigg[ \int \log \Big( \frac{\rd \calL}{\rd \tilde{\calL}} \Big) \rd  \calL^* \bigg].
\end{align}  
Performing importance sampling again yields
\begin{align}\label{eq:min_theta_x}
    \hat{\vTheta}^{*} &= \argmin_{\hat{\vTheta}} \bigg[\int \log \Big(\frac{\rd \calL}{\rd \tilde{\calL}} \Big)  \frac{\rd  \calL^*}{\rd \calL} \frac{\rd  \calL}{\rd \tilde{\calL}} \rd \tilde{\calL} \bigg].
\end{align}
The proposed iterative approach performs episodic reinforcement with respect to a loss function in order to optimize \cref{eq:min_theta_x}. Define the loss function as
\begin{align}\label{eq:def_loss_function}
    L(\hat{\vTheta}) &:= \Eb_{\tilde{\calL}}  \Bigg[\log \Big(\frac{\rd \calL}{\rd \tilde{\calL}} \Big)  \frac{\rd  \calL^*}{\rd \calL} \frac{\rd  \calL}{\rd \tilde{\calL}} \Bigg]
\end{align}
Plugging \cref{eq:Gibbs} and \cref{eq:RN} into \cref{eq:def_loss_function} yields
\small\begin{align}\label{eq:loss_function}
    L(\hat{\vTheta}^{(k)}) =  \Eb_{\tilde{\calL}}  \vast[ \frac{\exp( - \rho \tilde{J})}{\Eb_{\tilde{\calL}} \big[ \exp( - \rho \tilde{J}) \big]} \bigg( &-\sqrt{\rho} \calN(\hat{\vTheta}^{(k)}) - \frac{\rho}{2} \calP(\hat{\vTheta}^{(k)}) \bigg) \vast],
\end{align}\normalsize
where $\tilde{J}= \tilde{J}(Z_0^T, \hat{\vTheta}^{(k)})$ is defined as
\begin{equation}\label{eq:importance_sampled_cost}
    \tilde{J}(Z_0^T, \hat{\vTheta}^{(k)}) := J(Z_0^T) + \frac{1}{\sqrt{\rho}}\calN(\hat{\vTheta}^{(k)}) +\frac{1}{2}\calP(\hat{\vTheta}^{(k)}),
\end{equation}
and $J(Z_0^T)$ is a state cost evaluated over the state trajectory $Z_0^T$. For reaching tasks, $J(Z_0^T)$ is typically a weighted quadratic penalization of the 1-norm distance to the goal state.

This loss function compares sampled trajectories by evaluating them on the exponentiated $\tilde{J}$ performance metric. The importance sampling terms $\calN$ and $\calP$, which appear in $\tilde{J}$ add a quadratic control penalization term and a mixed control noise term. In the Loss function, they serve as weights for the exponentiated cost trajectories. For convenience, we denote the exponentiated cost term as $\calE :=  \Eb_{\tilde{\calL}} \big[ \exp( - \rho \tilde{J}) \big]^{-1}\exp( - \rho \tilde{J})$.

Recall, that the nonlinear policy $\vPhi$ is a functional mapping into Hilbert space $\calH$. This is kept general for derivation purposes, however it implies that the nonlinear policy affects controls each element of an infinite vector $Z \in \calH$. In a realistic context, there may not be control over how the actuation affects the field state $Z$.  A more realistic, but less general representation refines the policy as
\begin{equation}\label{eq:finite_actuation}
    \vPhi(t,Z,\vx;\vTheta^{(k)}) = \vm(\vx)^\top \varphi(Z ; \vTheta^{(k)} ),
\end{equation}
where $\vm(\vx):D^N \rightarrow \Rb^N \times \calH$ represents the effect of the actuation from $N$ actuators on the infinite-dimensional field. This function can take many forms, however in most applications it is either a Gaussian-like exponential with mean centered at the actuator locations, or an indicator function for actuation. 

In \cref{eq:finite_actuation}, $\varphi(X; \vTheta^{(k)}): \calH \rightarrow \Rb^N$ is a policy network with $N$ control outputs representing $N$ distributed (or boundary) actuators. Note that as desired, the tensor contraction given on the right hand side of \cref{eq:finite_actuation} produces a vector in $\calH$. Splitting the actuation function from the control signal is also desired because we ultimately wish to use a finite input, finite output policy network for the function $\varphi(X; \vTheta^{(k)})$. The inner product terms become
\begin{align}
    \calN(\hat{\vTheta}^{(k)}) &= \int_{0}^{T}\big\langle\vm(\vx)^\top \varphi(Z ; \vTheta^{(k)} ), \rd W(s)\big\rangle \label{eq:noise_inner_product} \\
    \calP(\hat{\vTheta}^{(k)}) &= \int_{0}^{T}||\vm(\vx)^\top \varphi(Z ; \vTheta^{(k)} )||^{2}\rd s \nonumber \\
    &= \int_{0}^{T}\big\langle \varphi(Z ; \vTheta^{(k)} ), \vM(\vx)\varphi(Z ; \vTheta^{(k)} ) \big\rangle \rd s \label{eq:policy_inner_product}
\end{align}
where $\vM(\vx) = \vm(\vx) \vm(\vx)^\top$.

%===============================================================================

\section{Approximate Discrete Optimization}

Performing spatio-temporal stochastic optimization in Hilbert spaces as described above maintains generality of the resulting loss function. In addition, it allows us to make use of the spatio-temporal noise process. Finally, we can apply the spatial inner product integration scheme described in \cite{Evans2019IDVRL}. Despite this, computations on a digital computer require spatial discretization. As a result, the portion of this optimization procedure dedicated to actuator co-design becomes a discrete optimization problem.

To see this more clearly, consider the \ac{1D} spatial continuum $D = [0,1]$ discretized into a 10 point \ac{1D} grid. Lets assume that an actuator is chosen to be placed at $x=0.25$. Even though the actuation function $m(x)$ may be Gaussian-like function, the majority of the actuation will be felt in between two grid points, namely 0.2 and 0.3. This problem is even more severe if the actuation function $m(x)$ is the indicator function, as there will be no actuation exerted on the field irrespective of the control signal magnitude. Denote the number of spatial discretization points as $J$ and the discretized problem domain grid as $\hat{D}$ composed of $J^3$ elements. The optimization problem becomes
\begin{equation} \label{eq:discrete_min_theta_hat}
    \begin{split}
        \min_{\vTheta, \vx}\;\; &L(\vTheta, \vx) \\
        \text{subject to}\;\; &\vx \in \hat{D}
    \end{split}
\end{equation}

This formulation is an accurate representation, yet limits gradient flow from the loss function back to the actuator design parameters. In order to maintain these gradients, we approximate \cref{eq:discrete_min_theta_hat} as follows. Define a one-to-one map $S:\hat{D} \rightarrow \Zb_+$, where $\Zb_+$ denotes positive integers. Applying the forward and inverse mapping produces a gradient-based parameter update of the form 
\begin{align}
    \vTheta^{(k+1)} &=  \vTheta^{(k)} - \gamma_\vTheta \nabla_{\vTheta}L(\vTheta^{(k)}, \vx^{(k)}) \label{eq:theta_update}\\
    \vx^{(k+1)} &= S^{-1}\bigg(R\Big(S\big(\vx^{(k)} - \gamma_\vx \nabla_\vx L(\vTheta^{(k)}, \vx^{(k)})\big)\Big)\bigg) \label{eq:x_update}
\end{align}
where $R(\cdot)$ simply rounds to the nearest integer, $\gamma_\vTheta$ is the learn rate for the \ac{ANN} parameters, $\gamma_\vx$ is the learn rate for the actuator design parameters, $\nabla_\vTheta$ denotes the gradient with respect to $\vTheta$, and $\nabla_\vx$ deontes the gradient with respect to $\vx$. This approach allows us to leverage well-known backprop-based algorithms such as ADA-Grad~\cite{duchi2011adaptive} and ADAM~\cite{kingma2014adam}.

%===============================================================================

\section{Algorithm and Network Architecture}\label{sec:algorithm}

As discussed previously, implementation of the above framework requires spatial and temporal discretization of the \acp{SPDE} discussed in the first section. With this in mind, we choose an \ac{ANN} for our nonlinear policy $\varphi(Z; \vTheta^{(k)})$. In this work we exclusively use \acp{FNN} for all of our experiments, and use physics-based models of each \acp{SPDE} to generate training data. Given that the proposed framework is semi-model-free, real system data can seamlessly replace the physics-based model as described in \cite{Evans2019IDVRL}. We only need prior knowledge of the flavor of noise and the actuator dynamics (i.e. $\vm(\vx)$). Also described is a sparse method for spatial integration that we apply here to maintain computational and memory efficiency. 

The resulting algorithm, which we call the \ac{ADPL} algorithm is shown in \cref{Algorithm1}. The inputs are time horizon ($T$), number of iterations ($K$), number of rollouts ($R$), initial state ($Z_0$), number of actuators ($N$), noise variance ($\rho$), time discretization ($\Delta t$), actuator variance ($\sigma_\mu$), initial network parameters ($\vTheta^{(0)}$), initial actuator locations ($\vx^{(0)}$), policy learn rate ($\gamma_\vTheta$), and actuator location learn rate ($\gamma_\vx$). For more information on $SampleNoise()$, refer to~\cite[Chapter 10]{lord_powell_shardlow_2014}.

In addition to obtaining gradients from any variant of \ac{SGD} via backprop paths, $GradientStep$ for actuator parameters also adds the current gradient to the previous gradient if the actuator location has not moved by the current gradient step. This is useful because when the actuators are near their optimal value, the gradients become very small, preventing the actuator locations from reaching optimal values. This could also be achieved by heuristically changing the learning rate $\gamma_\vx$. Also, a quadratic cost term was added to $L$ in order to penalize actuators from leaving the spatial domain, with a large coefficient to ensure this condition was rarely, if ever violated.

Note that there are separate learning rates for the policy and the actuators. This is because in practice the authors found that the optimization landscape is typically much more shallow for the actuator design than for the policy parameters. For most of the experiments, the actuator placement learn rate $\gamma_\vx$ was about 30 times larger than the policy network learn rate $\gamma_\vTheta$.

\begin{algorithm}[t]
 \caption{Actuator Design and Policy Learning}
 \begin{algorithmic}[1]
 \State \textbf{Function:} \textit{$\vTheta^* =$ \textbf{OptimizePolicyNetwork}($T$,$K$,$R$,$Z_0$,$N$,$\rho$, $\Delta t$,$\mu$, $\sigma_{\mu}$,$\vTheta^{(0)}$, $\vx^{(0)}$, $\gamma_\vTheta$, $\gamma_\vx$)}
 \For {$k=1 \; \text{to} \; K$}
  \State Compute $\vm(\vx), M(\vx)$ $\forall$ $\vy \in D$
 \For {$r=1 \; \text{to}\; R$}
 \For{$t=1 \;\text{to}\; T$}
     \State $\rd W_t \gets SampleNoise()$
     \State $Z_t \gets Propagate(Z_{t-1},\vTheta^{(k)}, \rd W_t)$ via \cref{eq:EB_Hilbert_form}
     \State $u_t \gets SparseForwardPass(\vTheta^{(k)},Z_t)$
 \EndFor
 \EndFor
 \State $J \gets StateCost(Z_0^T)$
 \State $N\gets \calN\big(u_0^T, \rd W,  \vm(\vx)\big)$ via \cref{eq:noise_inner_product}
 \State $P \gets \calP\big(u_0^T, M(\vx)\big)$ via \cref{eq:policy_inner_product}
 \State $E \gets \calE(J, N, P)$ as in \cref{eq:loss_function}
 \State $L \gets ComputeLoss(P,N,E)$ via \cref{eq:importance_sampled_cost}
 \State Compute $\nabla_\vTheta(L)$ via backprop
 \State Compute $\nabla_\vx(L)$ via backprop
 \State $\vTheta^{(k+1)} \gets GradientStep(L, \gamma_\vTheta)$ via \cref{eq:theta_update}
 \State $\vx^{(k)} \gets GradientStep(L, \gamma_\vx)$ via \cref{eq:x_update}
 \EndFor
 \end{algorithmic}
 \label{Algorithm1}
\end{algorithm}

%===============================================================================

\section{Simulation Results and Discussion}

\label{sec:simulation}
We applied our approach to four simulated \ac{SPDE} experiments to simultaneously place actuators and optimize a policy network. Each experiment used less than 32 GB RAM, and was run on a desktop computer with a Intel® Xeon(R) 12-core CPU with a NVIDIA GeForce GTX 980 GPU. Our code was written to operate inside a Tensorflow graph~\cite{tensorflow} to leverage rapid static graph computation, as well as sparse linear algebra operations used by $SparseForwardPass$. The first two experiments involved a reaching task, where the \acp{SPDE} is initialized at a zero initial condition over the spatial region, and must reach certain values at pre-specified regions of the spatial domain. The last two experiments involved a suppression task, where some non-zero initial condition must be suppressed on desired regions.

The data that was used for training was generated by a spatial central difference, semi-implicit time discretized version of each \ac{SPDE}. These schemes are described in detail in~\cite[Chapter 3 \& 10]{lord_powell_shardlow_2014}. Each experiment had all actuators initialized by sampling from a random distribution on $[0.4a, 0.6a]$, where $a$ denotes the spatial size. For 3500 iterations of our algorithm, run times for the most complicated system--the Euler-Bernoulli equation--were about 15 hours.

The first experiment was a temperature reaching task on the \ac{1D} Heat equation 
% given by
% \begin{equation}
%     \partial_t u = \varepsilon \partial_{xx} u
% \end{equation}
with homogeneous, $T=0$ Dirichlet boundary conditions, and is depicted in \cref{fig:heat}. The task was to raise the temperature at regions specified in green to specified values depicted in the figure. The algorithm was run for 3000 iterations. 
% The converged actuator locations were located at $[0.188, 0.422, 0.828]$.

The next experiment was a velocity reaching task on the Burgers equation with non-homogenous Dirichlet boundary conditions, and is depicted in \cref{fig:burgers}. The Burgers equation has a nonlinear advection term, which produces an apparent rightward motion. The algorithm was run for 3500 iterations, and was able to take advantage of the advection for actuator placement in order to solve the task with lower control effort.

The heat equation is a pure diffusion \ac{SPDE}, while the Burgers equation shares the diffusion term with the Heat equation with an added advection term. The results of the Heat and Burgers experiments show actuator locations that take advantage of the natural behavior of each \acp{SPDE}. In the case of the Heat equation, actuators are nearby the desired regions such that the temperature profile can reach a flat peak of the diffusion at the desired profile. In the case of the burgers equation, the advection pushes towards the right end of the space, thus forming a wave front that develops at the right end, but leaves the left end dominated by the diffusion term. This is again reflected in the placement of actuators. The first actuator is nearby the desired region just as the actuators in the Heat \ac{SPDE}, while two of the actuators between the center and the right region are located to be able to control the amplitude and shape of the developing wave front so as to produce a flat peak that aligns with the desired region at the desired velocity. The central desired region is flanked on both sides by actuators that are nearly equidistant, in order to produce another desired flat velocity region at this location.

{\centering
\begin{figure*}[t!]
\begin{multicols}{2}
    
    \begin{subfigure}[h!]{1.0\textwidth}
    \hspace{-1.0cm}\includegraphics[width=0.65\textwidth]{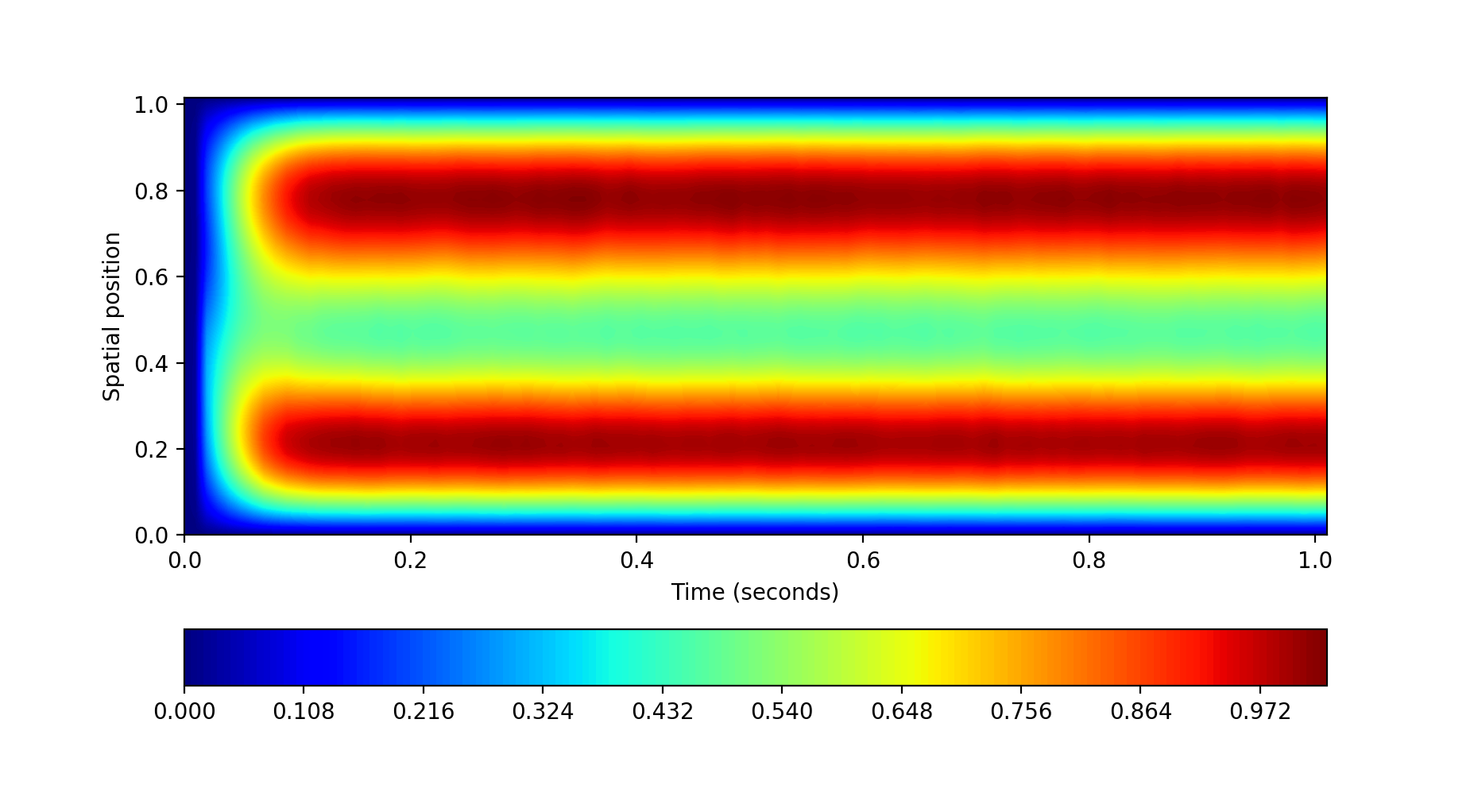}
    \end{subfigure}
    
    \begin{subfigure}[h!]{1.0\textwidth}
     \hspace{0.5
     cm}\includegraphics[width=0.5\textwidth]{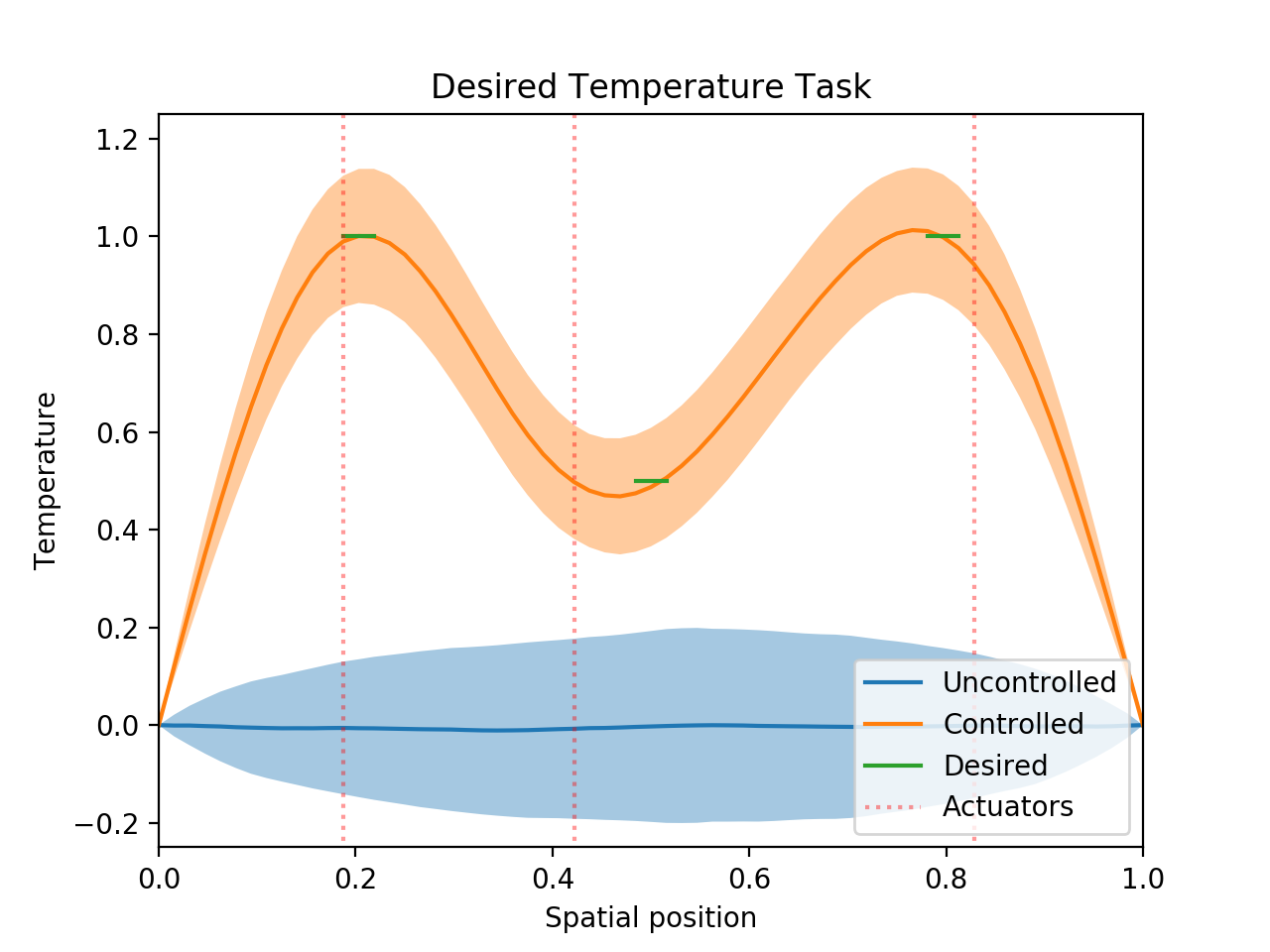}
    \end{subfigure}
    \end{multicols}
    \vspace{-0.5cm}
    \caption{Heat Equation Temperature Reaching Task. (left) controlled contour plot where color represents temperature, (right) final time snapshot comparing to the uncontrolled system. Mean trajectories are represented with a solid line, while a 2$\sigma$ standard deviation is represented with a shaded region.}
    \label{fig:heat}
\end{figure*}}

{\centering
\begin{figure*}[t!]
\begin{multicols}{2}
    
    \begin{subfigure}[h!]{1.0\textwidth}
    \hspace{-1.0cm}\includegraphics[width=0.65\textwidth]{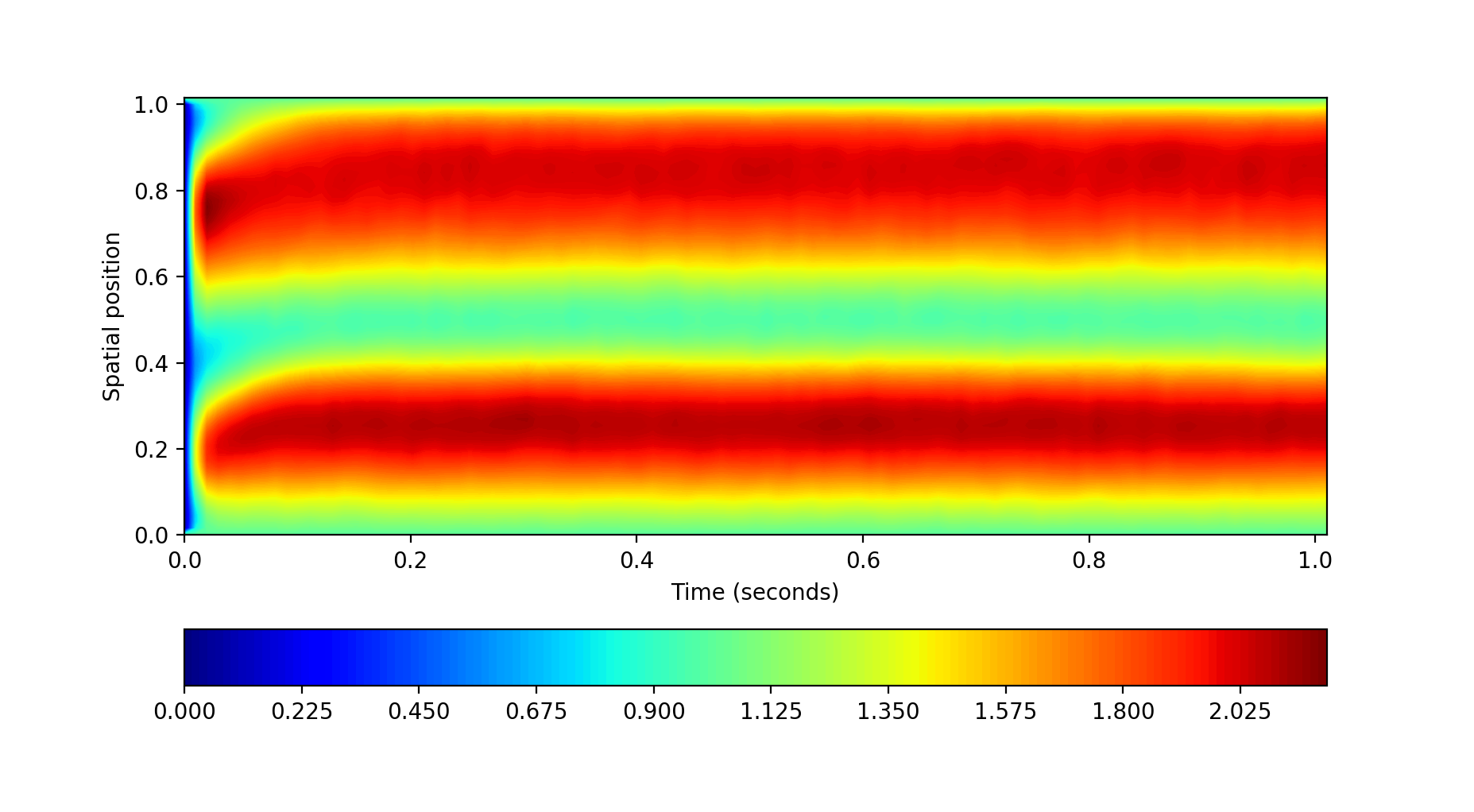}
    \end{subfigure}
    
    \begin{subfigure}[h!]{1.0\textwidth}
     \hspace{0.5
     cm}\includegraphics[width=0.5\textwidth]{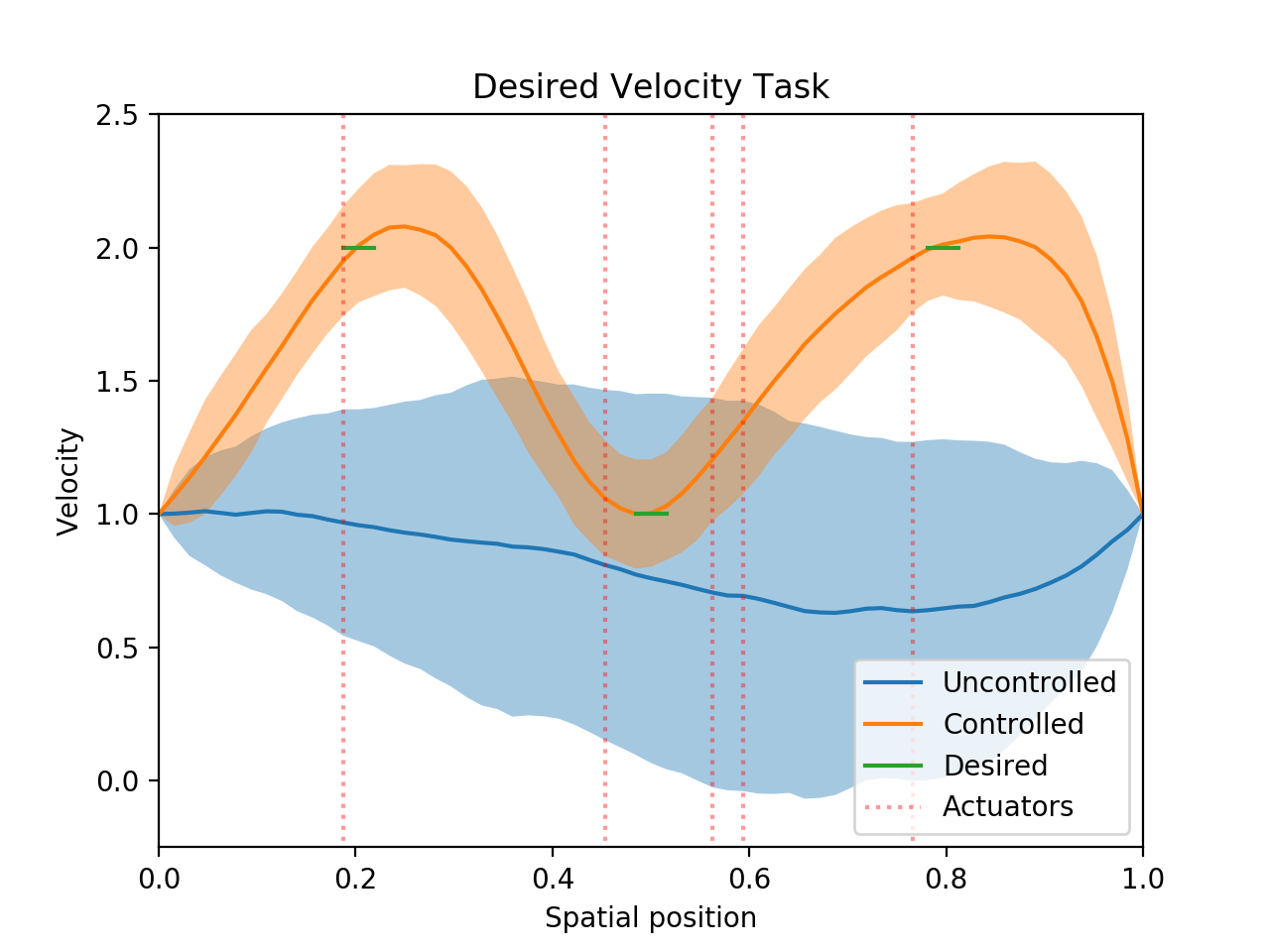}
    \end{subfigure}
    \end{multicols}
    \vspace{-0.5cm}
    \caption{Burgers Velocity Reaching Task. (left) controlled contour plot where color represents velocity, (right) final time snapshot comparing to the uncontrolled system. Mean trajectories are represented with a solid line, while a 2$\sigma$ standard deviation is represented with a shaded region.}
    \label{fig:burgers}
\end{figure*}}

{\centering
\begin{figure*}[!ht]
\begin{multicols}{2}
    
    \begin{subfigure}[h!]{1.0\textwidth}
    \hspace{-1.0cm}\includegraphics[width=0.65\textwidth]{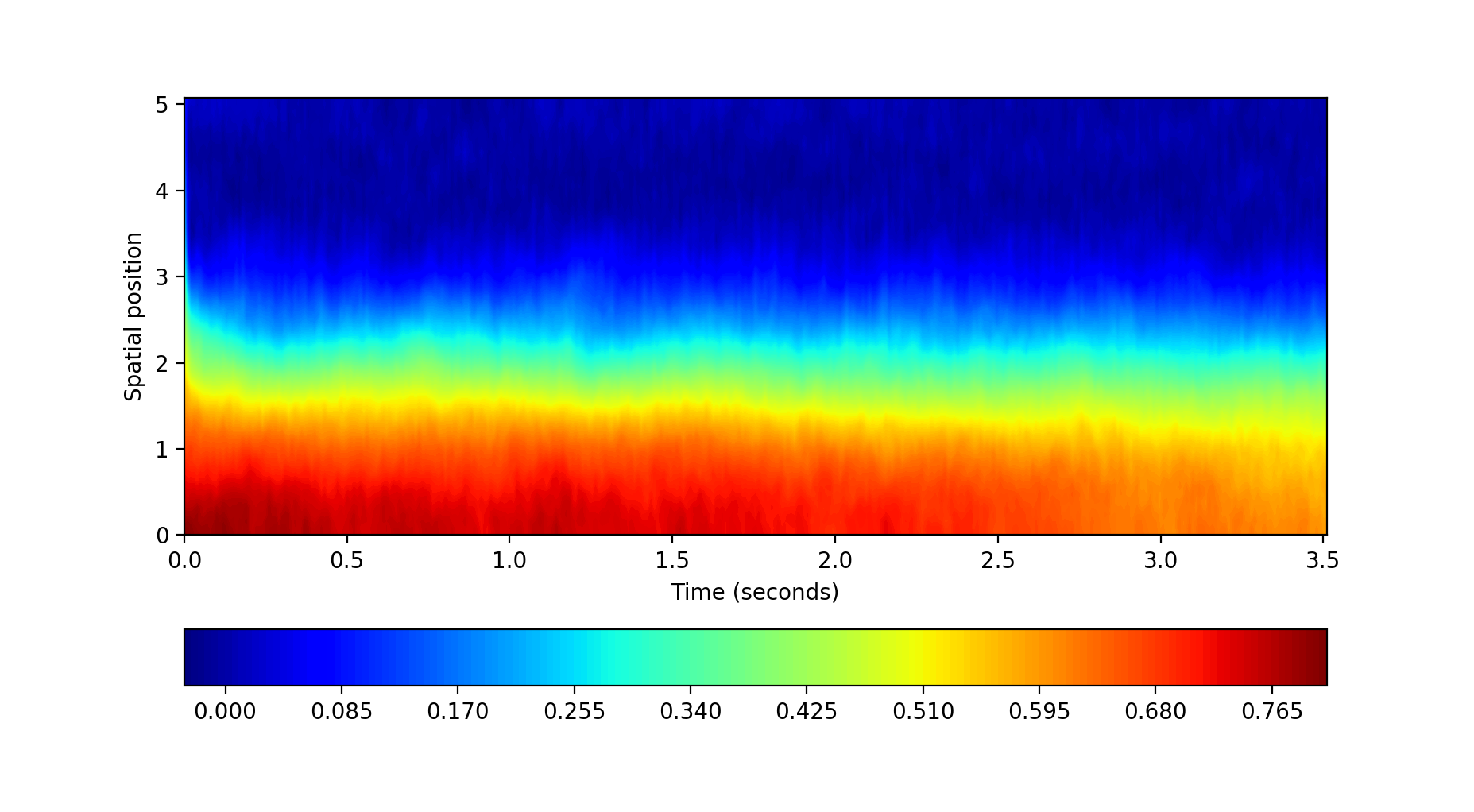}
    \end{subfigure}
    
    \begin{subfigure}[h!]{1.0\textwidth}
     \hspace{0.5
     cm}\includegraphics[width=0.5\textwidth]{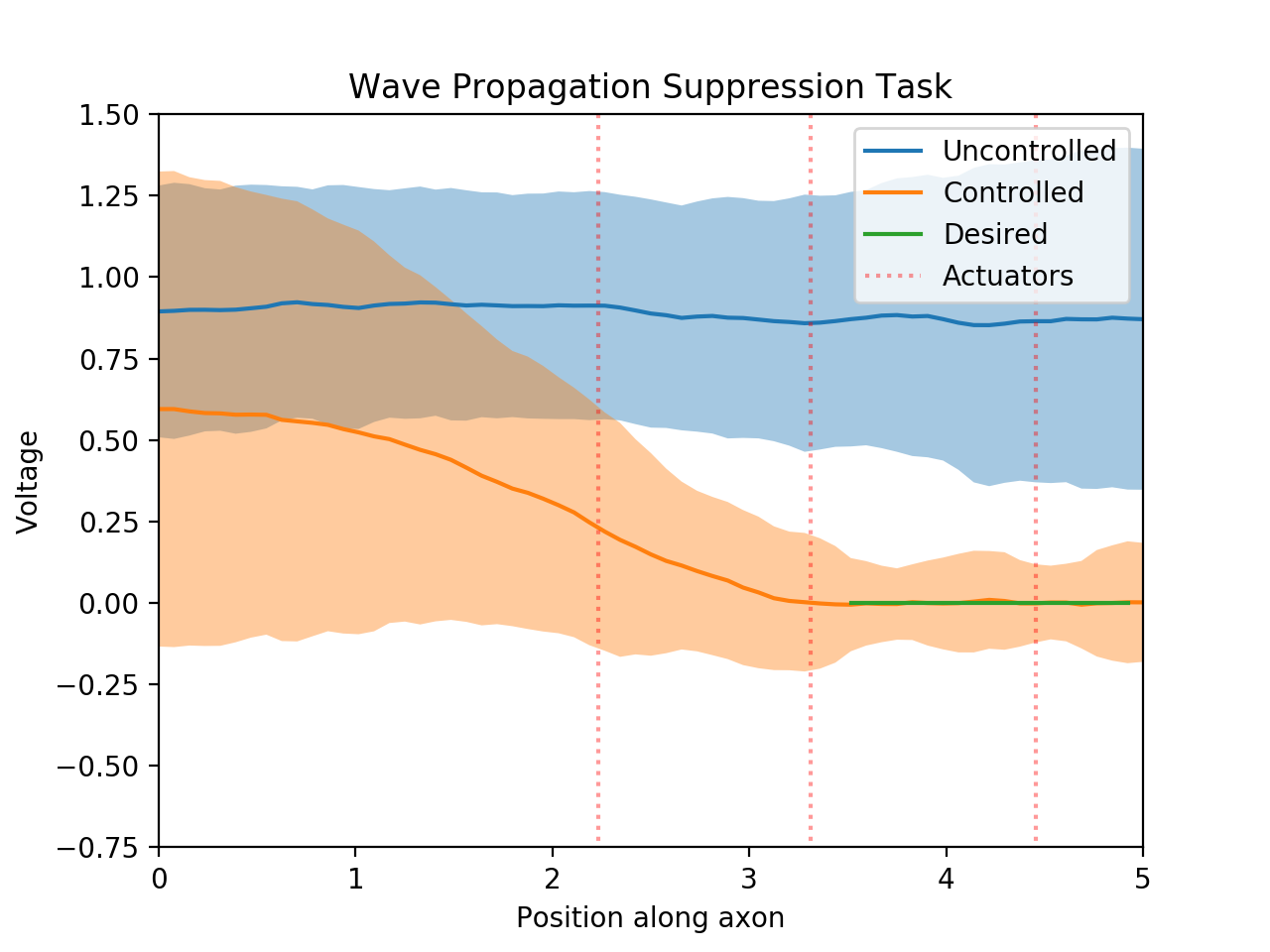}
    \end{subfigure}
    \end{multicols}
    \vspace{-0.25cm}
    \caption{Nagumo Suppression Task. (left) controlled contour plot where color represents voltage, (right) final time snapshot comparing to the uncontrolled system. Mean trajectories are represented with a solid line, while a 2$\sigma$ standard deviation is represented with a shaded region.}
    \label{fig:nagumo}
\end{figure*}}

{\centering
\begin{figure*}[t]
\begin{multicols}{3}
    
    \begin{subfigure}[h!]{1.0\textwidth}
    \hspace{-.35cm}\includegraphics[width=0.45\textwidth]{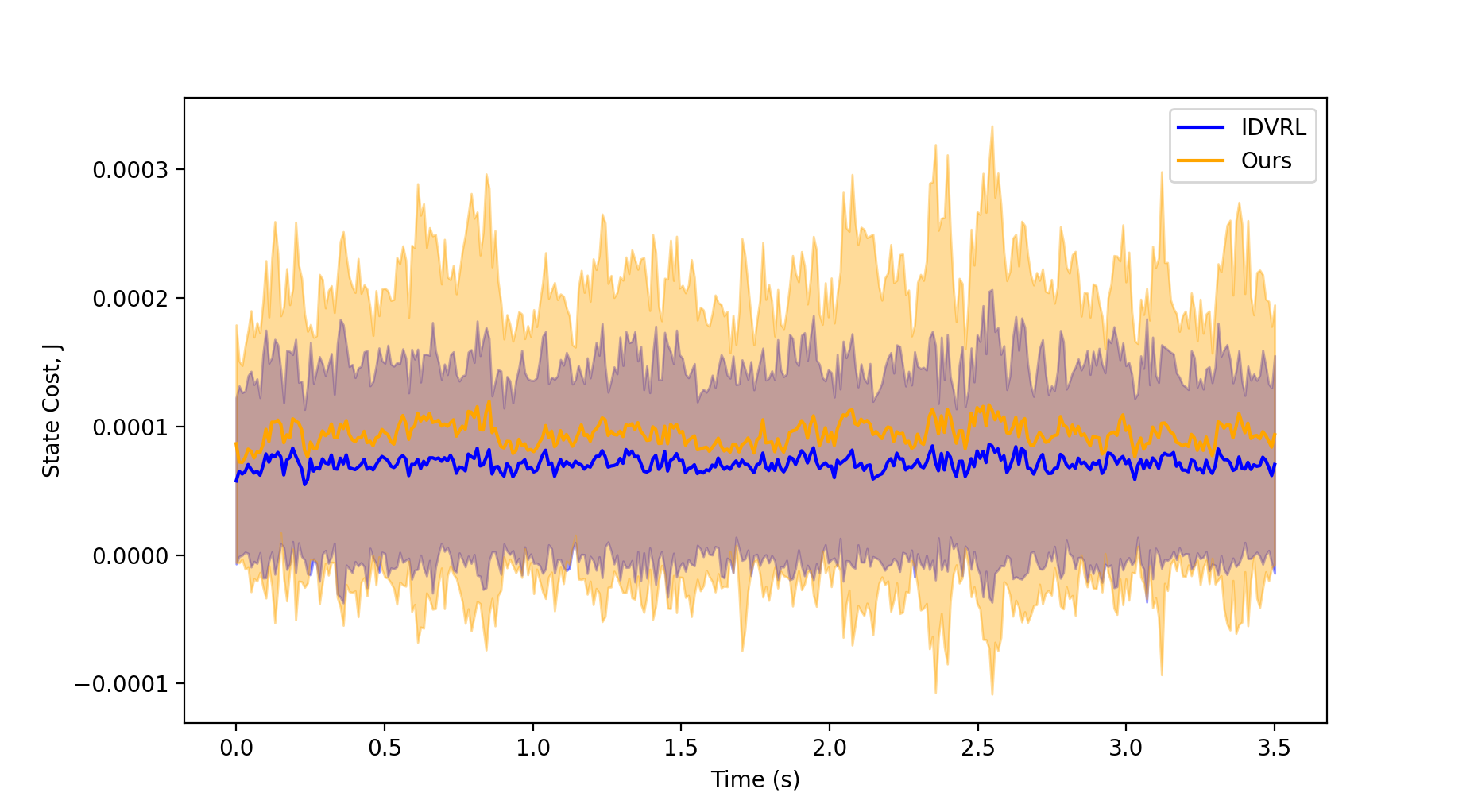}
    \end{subfigure}
    
    \begin{subfigure}[h!]{1.0\textwidth}
     \hspace{0.85cm}\includegraphics[width=0.329\textwidth]{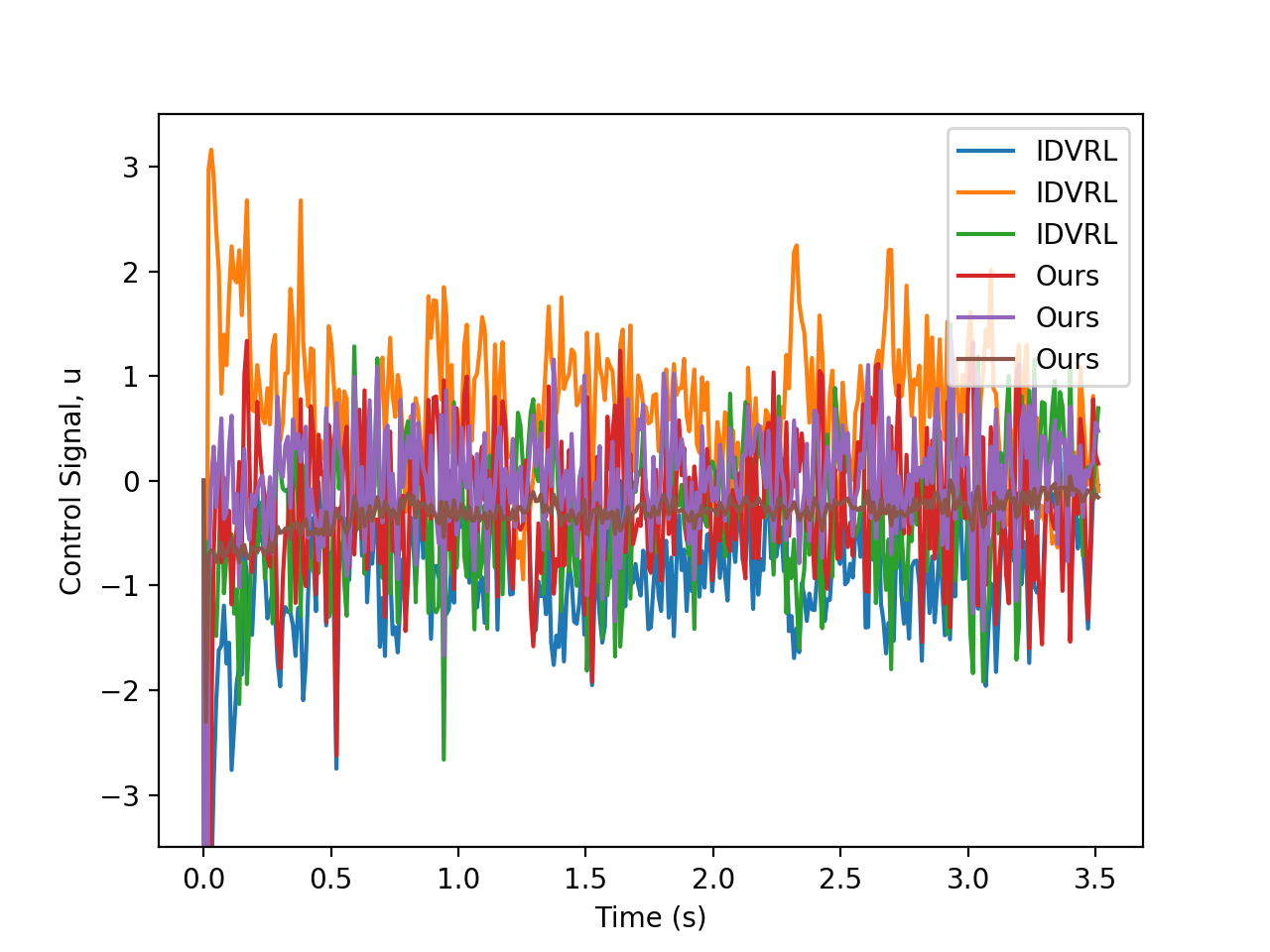}
    \end{subfigure}
    
    \begin{subfigure}[h!]{1.0\textwidth}
    \hspace{0.2cm}\includegraphics[width=0.329\textwidth]{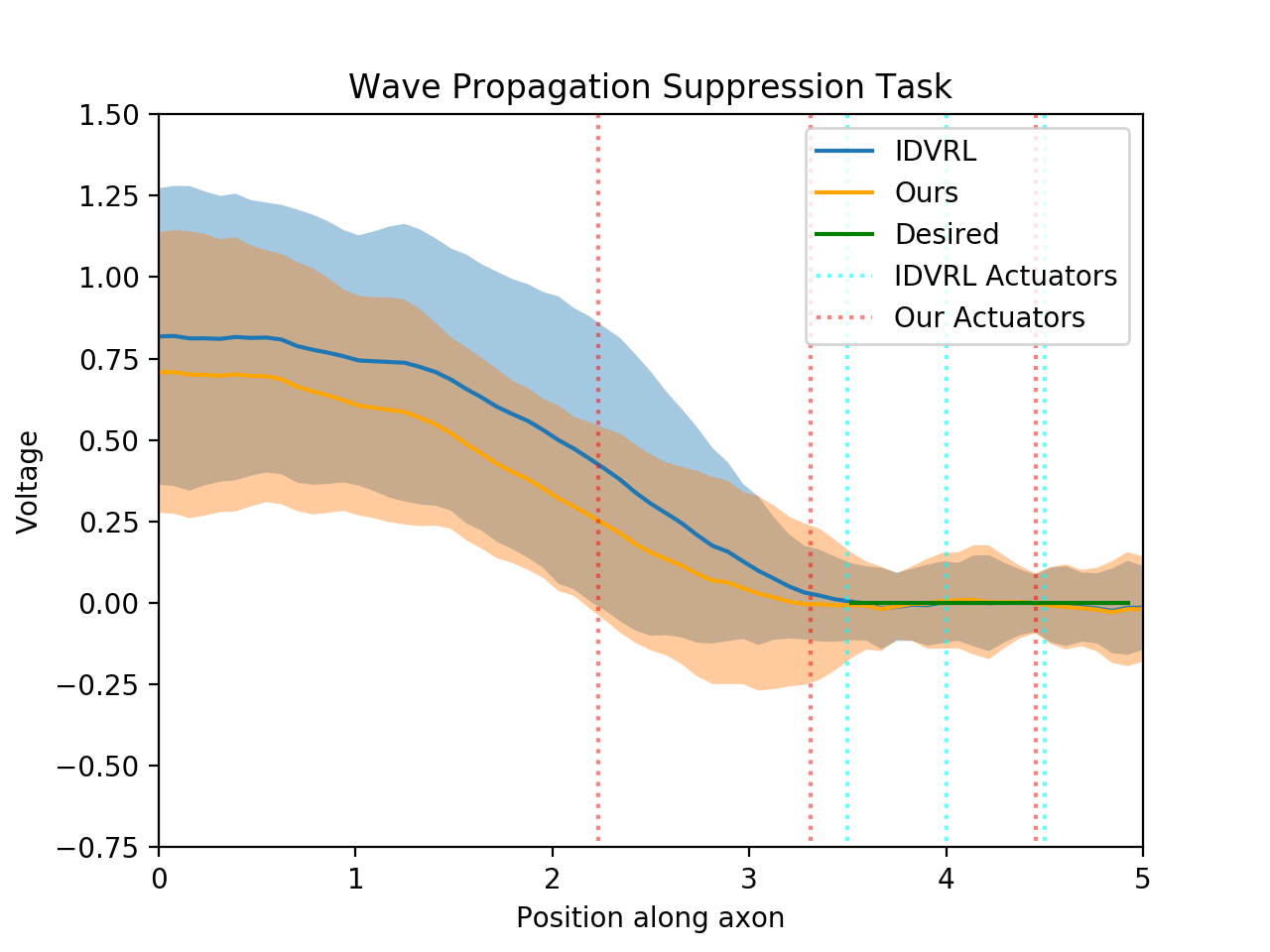}
    \end{subfigure}
    % Nagumo Suppression Task: Comparison between our approach and actuators placed by a human expert and policy optimization via IDVRL.}
    % \label{fig:IDVRL_comparison_snapshot}
    
    \end{multicols}
    \vspace{-0.25cm}
    \caption{Nagumo Suppression Task Comparison Plots. (left) controlled state cost plot, where solid lines denote mean, and shaded regions denote a $2 \sigma$ standard deviation, (center) Control signal comparison plot, where lines represent mean behavior, and (right) Final time snapshot comparing the actuators placed by our approach and actuators placed by a human expert with policy optimization by IDVRL.}
    \label{fig:nagumo_comparison}
\end{figure*}}

{\centering
\begin{figure*}[!h]
\begin{multicols}{2}
    
    \begin{subfigure}[h!]{1.0\textwidth}
    \hspace{-0.82cm}\includegraphics[width=0.614\textwidth]{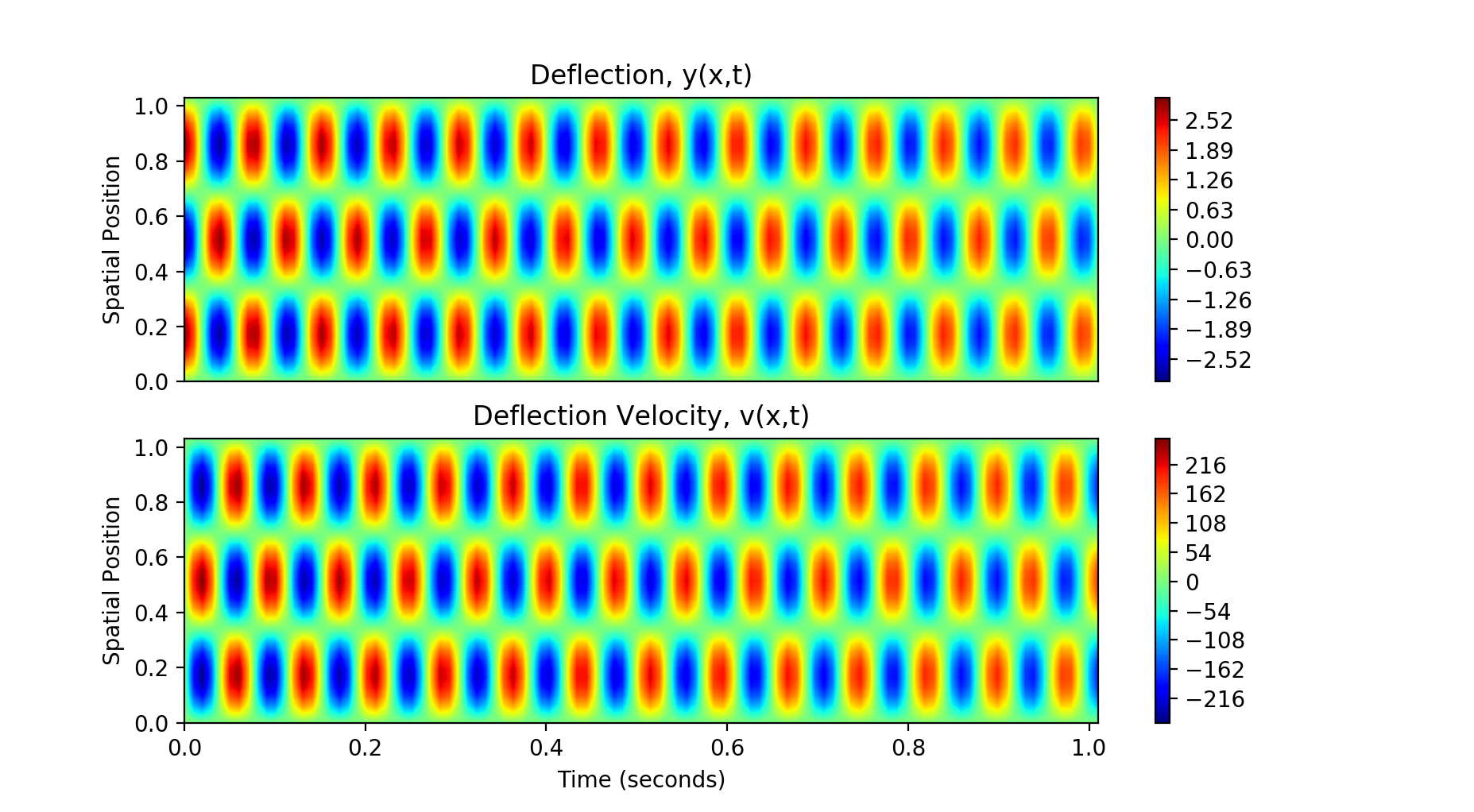}
    \end{subfigure}
    
    \begin{subfigure}[h!]{1.0\textwidth}
     \hspace{-0.65cm}\includegraphics[width=0.614\textwidth]{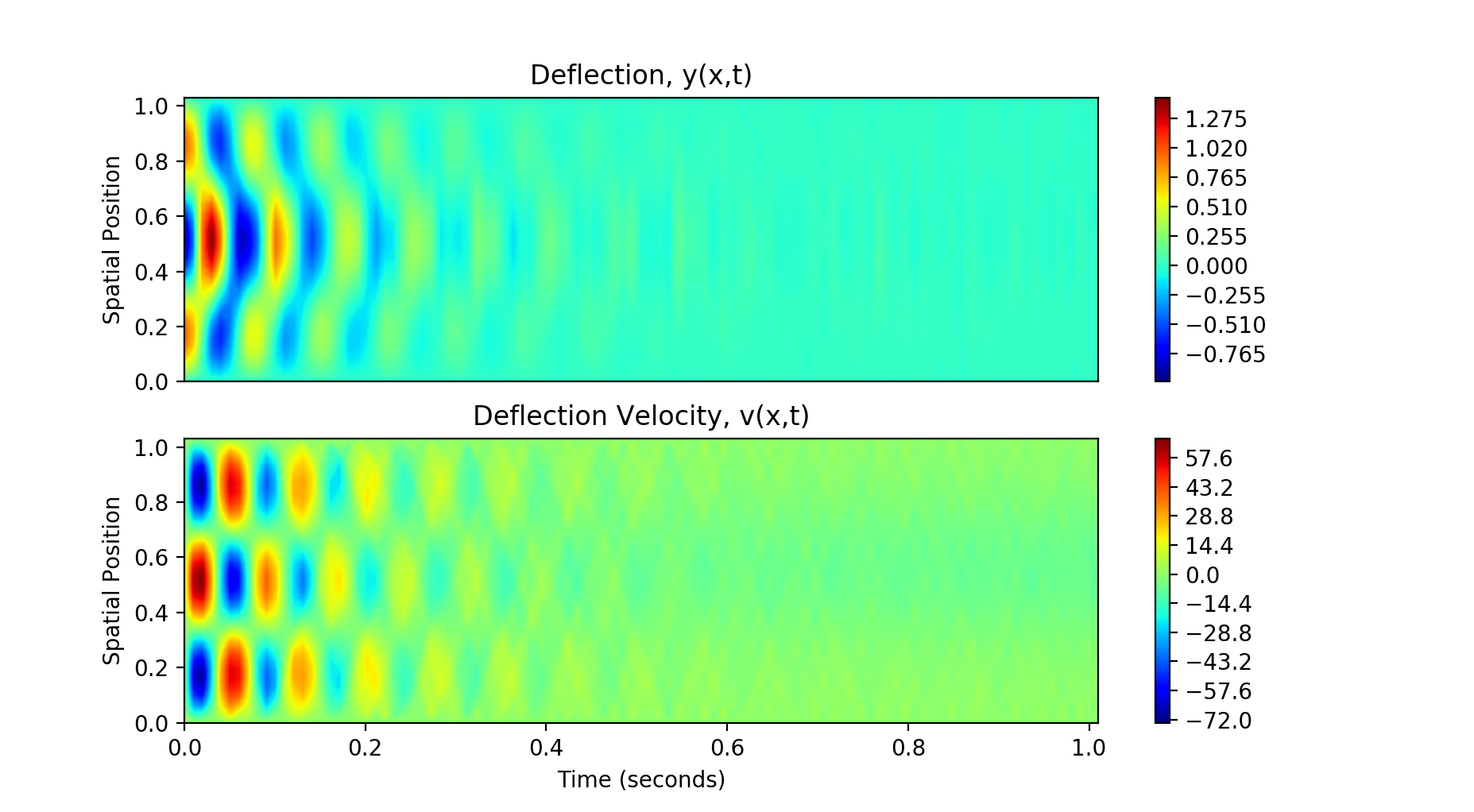}
    \end{subfigure}
    \end{multicols}
    \vspace{-0.25cm}
    \caption{Euler-Bernoulli Suppression Task. (left) Uncontrolled contour plot (right) Controlled contour plot. In both plots, color represents deflection on top, and deflection velocity on bottom.}
    \label{fig:EB}
\end{figure*}}

The third experiment was a voltage suppression task on the Nagumo equation with homogeneous Neumann boundary conditions, and is depicted in \cref{fig:nagumo}. The task was to suppress an initial voltage on the left end, that without intervention propagates toward the right end, as shown by the uncontrolled trajectories. The Nagumo equation is composed of a diffusive term and a 3rd-order nonlinearity, making this equation the most challenging from a nonlinear control perspective. Despite this, our approach was able to simultaneously place actuators and provide control such that the task was solved. The algorithm was run for 2000 iterations, and demonstrates actuator placement optimization that takes advantage of the natural system behavior. This task was also the most challenging due to the significantly longer planning horizon of 3.5 seconds, as compared to the 1.0 second planning horizon of all the other experiments.

In order to validate our proposed approach, we compared the actuator locations that the algorithm found after optimization to the actuator locations that were hand placed by a human expert for the simulated experiments conducted in related work \cite{Evans2019IDVRL}. To have a valid comparison, we ran the IDVRL algorithm for both sets of actuator locations. Figure \ref{fig:nagumo_comparison} reports these results. The left figure shows that the state costs for each is almost identical. Note that the scale here is $10^{-4}$. The center figure shows the control signals for each actuator, for each method, and demonstrates that for almost identical state cost values, the control effort for each actuator with our approach is lower on average. The calculated average control signal magnitudes for IDVRL are 3.3 times higher than our method. The third plot shows the voltage profile at the final time. We hypothesize that the lower control effort is due to the control over the shape of the spatially propagating signal, enabling it to have a smoother transition into the desired region. While the penalty of this actuator placement is a higher variance on the desired region, the choice appears correct given the result.

The final task was an oscillation suppression task on the Euler-Bernoulli equation with Kelvin-Voigt damping given in \cref{eq:EB_SPDE}, and is depicted in \cref{fig:EB}. As shown, the initial condition prescribes spatial oscillations, that then oscillate temporally. The second-order nature of the system creates offset and opposite oscillations in the velocity profile, that in turn produce offset and opposite oscillations in the position profile. Without interference, the oscillations proceed over the entire time window. As shown on the right, our approach successfully suppresses these oscillations, which die out quickly under the given control policy. In this experiment, the actuators did not leave the initialized actuator placement region $[0.4, 0.6]$ perscribed for all experiments.

The Euler-Bernoulli oscillation suppression task is in fact very challenging and complex. Producing a control signal at an actuator location that is in phase with the velocity oscillations will amplify the oscillations, leading to a divergence. The actuator location and control signal from the policy network must work in concert to produce a control signal out of phase with the velocity that matches its frequency, which is time varying due to control, as shown on the right side of \cref{fig:EB}.

Each of the above experiments has its challenges and in most cases the spatio-temporal problem space produces a joint policy optimization and actuator co-design problem that is littered with local minima. These experiments demonstrate that the proposed approach can jointly optimize a policy network and actuator design. These results and the overall performance of the algorithm are indicative that this approach may enable actuator design on problem spaces where a human has little to no prior knowledge to rely on in attempting to solve the problem by hand.

%===============================================================================

\section{Conclusion and Future Directions}
\label{sec:conclusion}

This work presents a framework for joint policy optimization and actuator co-design. We contribute a novel mathematical tool required for a measure theoretic treatment of second order \acp{SPDE}, namely the change of measure and associated proof. We also provide a mechanism for handling discrete parameter optimization with a gradient-based approach. We demonstrate the resulting algorithm on four different \acp{SPDE}, each with their own challenges and complexities. The last of which is the Euler-Bernoulli \ac{SPDE} which connects back with our goal of establishing capabilities for the further development of soft-body robotics. 

The presented approach is a new way of performing optimization and can lead to many applications in soft robotics, soft materials, morphing and continuum mechanism. The results are encouraging to the authors. We plan to scale  this approach to systems with higher spatial dimensionality (i.e. 2D and 3D), as well as investigate more complex forms of the above systems that are relevant to a future of soft robotics.
%===============================================================================

\bibliographystyle{ieeetran}
\balance
\bibliography{References}

%===============================================================================

\newpage
\section*{Supplementary Information}
\beginsupplement

\section{Additional Information on Simulations}
The following are details on each of our simulated experiments which will help in reproducing our results. In every simulated experiment we used spatial central difference discretization and semi-implicit time discretization. Our time discretization size was $\Delta t = 0.01$ throughout all experiments. All experiments initialized the actuator locations by sampling a uniform distribution on $[0.4a, 0.6a]$ at the policy network initialization phase. All experiments utilized \ac{FNN} for the nonlinear policy $\varphi(h;\Theta)$, with two hidden layers and each neuron applying ReLU activations. The number of neurons per layer matched the state dimension after discretization, which ended up being 64 for all experiments. All policy network weights were initialized with the Xavier initialization \cite{glorot2010understanding}. In every experiment the function $\vm(\vx)$ was modeled as a Gaussian-like exponential function with the means co-located with the actuator locations. Each experiment added a penalization term if actuators left the spatial region, with a coefficient of $10^3$. Each experiment considered a cost function of the form
\begin{equation}\label{supeq:cost}
    J := \sum_{t} \sum_{x} \;\kappa \big(h_{\text{actual}}(t,x) - h_{\text{desired}} (t,x)\big)^2 \cdot \mathbbm{1}_{S}(x)
\end{equation}
where $\mathbbm{1}_{S}(x)$ is defined by
\begin{equation} \label{supeq:1Dheat_indicator}
\mathbbm{1}_S(x) :=
\begin{cases}
1,  \quad \text{if } x \in S  \\
0, \quad  \text{otherwise},
\end{cases}
\end{equation}
where $S$ is the spatial subregion on which the desired profile is defined. These desired spatial regions were different for each experiment and are detailed below.

\subsection{Heat SPDE Reaching Task}

The controlled \ac{1D} heat \ac{SPDE} in Hilbert space form with homogeneous Dirichlet boundary conditions is given by
\begin{equation}
    \begin{split}
    \rd h(t,x) &= \epsilon h_{xx}(t,x) \rd t + G(t,h)\bigg(\,\vm(\vx)^{\top} \, \varphi(h;\Theta) \rd t  + \frac{1}{\sqrt{\rho}} \rd W(t) \bigg)\\
    h(t,0) &= h(t,a) = 0 \\
    h(0,x) &= h_0(x)
    \end{split}
\end{equation}
where $\epsilon$ is the thermal diffusivity parameter and was set to 1 for this experiment. The task is to achieve a desired temperature profile at three regions along the spatial domain of length $a=1.0$ with three actuators. The spatial region was discretized into 64 equally spaced grid points. This experiment used a noise term $\rho=10$. Here, the variance of the $N=3$ actuator Gaussian-like functions was set to $\sigma^2 = (0.1a)^2$ for all actuators. In this experiment $G(t,h)$ is the identity operator. The desired regions were defined by
\begin{equation}\label{supeq:1Dheat_regions}
\begin{split}
S_1 &= \{x \in D \mid x \in [0.18a, 0.22a] \} \text{ is the left spatial region}, \\
S_2 &= \{x \in D \mid x \in [0.48a, 0.52a] \} \text{ is the central spatial region}, \\
S_3 &= \{x \in D \mid x \in [0.78a, 0.82a] \} \text{ is the right spatial region}.
\end{split}
\end{equation}
and the desired temperature was set to $T=1.0$ for $x\in S_1 \cup S_3$, and $T=0.5$ for $x \in S_2$.

The network was trained using the ADAM optimizer for $3000$ iterations with $200$ trajectory rollouts sampled from the Heat \ac{SPDE} model per iteration. Each trajectory was $1.0$ seconds long. The policy learn rate $\gamma_\vTheta$ was set to $10^{-3}$ and the actuator location learn rate $\gamma_\vx$ was set to $3\times10^{-2}$.

\subsection{Burgers SPDE Reaching Task}
The controlled \ac{1D} Burgers \ac{SPDE} in Hilbert space form with non-homogeneous Dirichlet boundary conditions is given by
\begin{equation} \label{supeq:BurgersSPDE}
\begin{split}
\rd h(t, x) + h h_x(t, x)\rd t &= \epsilon h_{xx}(t,x) \rd t + G(t,h)\bigg(\vm(\vx)^{\top} \, \varphi(h;\Theta)\rd t \frac{1}{\sqrt{\rho}} \rd W(t) \bigg) \\
h(t,0) &= h(t,a) = 1.0\\ 
h(0,x) &= 0, \; \forall x \in (0,a)
\end{split}
\end{equation}
where the parameter $\epsilon$ is the viscosity of the medium. Equation \eqref{supeq:BurgersSPDE} considers a simple model of a 1D flow of a fluid in a medium with non-zero flow velocities at the two boundaries. The goal again here is to reach certain flow velocities at three specified spatial regions in the spatial domain of length $a=1.0$. The spatial region was discretized into 64 equally spaced grid points. Here we consider  $G(t,h)$ as an identity operator. This experiment used a noise term $\rho=10$. In this experiment the $N=5$ Gaussian-like functions used a variance of $\sigma^2 = (0.1a)^2$ for every actuator. The scaling parameter $\kappa$ was set to 100. The desired regions were given by \cref{supeq:1Dheat_regions}. The desired values were given by $h_{\text{desired}} (t,x)= 2.0 \;m/s$  for $x\;\in S_1 \cup S_3$, which is at the sides, and $ h_{\text{desired}} (t,x)= 1.0 \;m/s$  for $x\;\in S_2$, which is in the central region.

The network was trained using the ADAM optimizer for $3500$ iterations with $100$ trajectory rollouts sampled from the Burgers \ac{SPDE} model per iteration. Each trajectory was $1.0$ seconds long. The policy learn rate $\gamma_\vTheta$ was set to $10^{-3}$ and the actuator location learn rate $\gamma_\vx$ was set to $3\times10^{-2}$.

\subsection{1D Nagumo SPDE Suppression Task}
\label{supsec:NagumoSPDE} The \ac{1D} Nagumo equation in Hilbert space form with homogeneous Neumann boundary conditions is given by
\begin{align}
\rd h(t,x) &= \epsilon h_{xx}(t,x)\rd t + h(t,x)\big(1-h(t,x)\big)\big(h(t,x)-\alpha\big) \rd t\nonumber \\
&\quad + G(t,h)\bigg( \vm(\vx)^{\top} \, \varphi(h;\Theta)\rd t  + \frac{1}{\sqrt{\rho}} \rd W(t) \bigg) \nonumber \\ 
h_x(t,0) &= h_x(t,a) = 0\\ 
h(0,x) &= \bigg(1+\exp\Big(-\frac{2-x}{\sqrt[]{2}}\Big)\bigg)^{-1} \nonumber ,
\end{align}
where the parameter $\alpha=-0.5$ determines the speed of a wave traveling down the length $a=5.0$ of the axon and $\epsilon=1.0$ determines the rate of diffusion. We consider $G(t,h)$ as an identity operator. The goal of this task was to suppress the activation voltage that traveled rightward without control intervention. The $N=3$ actuator Gaussian-like functions used a variance of  $\sigma^2 = (0.1a)^2$ for each actuator. Here, the noise term was set to $\rho=10.0$ The spatial domain was discretized using a grid of 64 points. The cost function coefficient $\kappa$ was chosen as $10^{-3}$. The experiment used one desired spatial region $S = [0.7a, 0.99a]$, which is at the right end. 

The network was trained using the ADAM optimizer for $2000$ iterations with $100$ trajectory rollouts sampled from the Nagumo \ac{SPDE} model per iteration. Each trajectory was $3.5$ seconds long. The policy learn rate $\gamma_\vTheta$ was set to $10^{-3}$ and the actuator location learn rate $\gamma_\vx$ was set to $5\times10^{-2}$.

\subsection{Euler-Bernoulli SPDE Suppression Task}

The controlled \ac{1D} Euler-Bernoulli \ac{SPDE} with homogeneous Dirichlet boundary conditions is given in the paper, but repeated here for completess
\begin{equation} \label{supeq:EB_SPDE}
\begin{split}
&\partial_{tt} y + \partial_{xx} \big( \partial_{xx}y + C_d \partial_{xxt} y \big) + \mu \partial_t y = \vPhi + \frac{1}{\sqrt{\rho}} \partial_t W(t) \\
&y(t,0) = y(t,a) = 0 \\
&y(0,x) = y_0 \\
&\partial_t y(0,x) = v_0 \\
&\partial_{xx}(t,0) + C_d \partial_{xxt}y(t,0) = 0 \\
&\partial_{xx}(t,a) + C_d \partial_{xxt} y(t,a) = 0
\end{split}
\end{equation}
This is a second order system in time. Similar to a typical treatment of second order \acp{SDE}, we define a velocity state $v := \partial_t y$. After rewriting \eqref{supeq:EB_SPDE} and lifting it into Hilbert spaces, we have
\begin{equation} \label{supeq:EB_Hilbert_form}
\begin{split}
    &\rd Z = A Z \rd t + G\Big( \vPhi(t,Z,\vx; \vTheta^{(k)})\rd t + \frac{1}{\sqrt{\rho}}\rd W(t) \Big), \\
    &y(t,0) = y(t,a) = 0 \\
&y(0,x) = y_0 \\
&v(0,x) = v_0 \\
&\partial_{xx}(t,0) + C_d \partial_{xx}v(t,0) = 0 \\
&\partial_{xx}(t,a) + C_d \partial_{xx}v(t,a) = 0
\end{split}
\end{equation}
where $Z$ contains the both the displacement $y$ and displacement velocity $v$, $A: \calH^2 \rightarrow \calH^2$ is the linear operator $A = [0 \;\; I; -A_0 \;\; -C_d A_0 - \mu I]$, $G: \calH \rightarrow \calH^2$, and $\rd W(t)$ is a Cylindrical Wiener process. The Kelvin-Voigt damping term $C_d$ was set to $10^{-4}$, and the viscous damping term $\mu$ was set to $10^{-3}$. This was to almost completely remove damping of the initialized oscillations throughout the entire time window. The initial displacement was set by the function 
\begin{equation}
    y_0 = \sin\Big(\frac{3 \pi x}{a}\Big), \quad \forall \; x \in [0, a]
\end{equation}
As shown in the uncontrolled contour plot (fig. 5 in the main paper), this initial condition generates oscillations over the spatial length $a = 1.0$ that evolve temporally. The position and velocity spaces were each spatially discretized into 32 equal points, for a total state space dimensionality of 64. In this experiment, the $N=8$ actuator Gaussian-like functions used a variance of $\sigma^2 = (0.2a)^2$. The noise term $\rho$ was set to 1.0 for this experiment. The cost function coefficient $\kappa$ was set to $3 \times 10^{-4}$, and the desired spatial subregion was defined over the entire space $S=[0, 1.0]$.

The network was trained using the ADAM optimizer for $3500$ iterations with $250$ trajectory rollouts sampled from the Euler-Bernoulli \ac{SPDE} model per iteration. Each trajectory was $1.0$ seconds long. The policy learn rate $\gamma_\vTheta$ was set to $10^{-3}$ and the actuator location learn rate $\gamma_\vx$ was set to $5\times10^{-2}$.

\end{document}